\documentclass[12pt]{article}

 \usepackage{amsmath,amssymb,amscd,amsthm,esint}
 
\usepackage{graphics,amsmath,amssymb,amsthm,mathrsfs}

\usepackage{graphics,amsmath,amssymb,amsthm,mathrsfs,amsfonts}

\usepackage{sidecap}
\usepackage{float}
\usepackage{extarrows}
\usepackage{booktabs}
\usepackage{verbatim}
\usepackage{hyperref}
\usepackage[usenames,dvipsnames]{xcolor}

\newtheorem{thm}{Theorem}[section]
\newtheorem{lemma}[thm]{Lemma}
\newtheorem{cor}[thm]{Corollary}

\theoremstyle{definition}
\newtheorem{remark}[thm]{Remark}

\def\XXint#1#2#3{{\setbox0=\hbox{$#1{#2#3}{\int}$}
         \vcenter{\hbox{$#2#3$}}\kern-.5\wd0}}

\def\R{\mathbb{R}}

\def\e{\varepsilon}

\def\N{\mathbb{N}}
\def\loc{\text{loc}}

\numberwithin{equation}{section}

\begin{document}

\title{Compactness and Large-Scale Regularity  \\ for  Darcy's Law}

\author{
Zhongwei Shen\thanks{Supported in part by NSF grant DMS-1856235.}}
\date{}
\maketitle

\begin{abstract}

This paper is concerned with the quantitative homogenization of the steady Stokes equations with the Dirichlet condition 
in a periodically perforated domain.
Using a compactness method, we establish the large-scale  interior $C^{1, \alpha}$ and Lipschitz estimates for the velocity
as well as the corresponding estimates for the pressure.
These estimates, when combined with the classical regularity estimates for the Stokes equations,
yield the uniform  Lipschitz estimates.
As a consequence, we also obtain the uniform $W^{k, p}$ estimates for $1<p<\infty$.
 
\medskip


\medskip

\noindent {\it MR (2020) Subject Classification}: 35Q35;  35B27; 76D07.

\end{abstract}


\section{Introduction}\label{section-1}

In this paper we continue the study of the quantitative  homogenization of the steady Stokes equations
for an incompressible viscous fluid,
\begin{equation}\label{Stokes}
\left\{
\aligned
- \e^2 \mu  \Delta u_\e + \nabla p_\e & = f,\\
\text{\rm div}(u_\e) & =0,
\endaligned
\right.
\end{equation}
with a no-slip (Dirichlet) boundary  condition on solid pores,
 in a periodically perforated  domain in $\R^d$, $d\ge 2$.
In \eqref{Stokes}, $\mu>0$ is the viscosity constant, and
we have normalized the velocity vector by a factor $\e^2$, where $\e>0$ is the period.
It is well known 
  that as $\e\to 0$, the effective equations for \eqref{Stokes}
are  given by a Darcy law \cite{SP, Tartar-80, Allaire-89, LA-1990, Allaire-91a, Mik-1991, Allaire-1997}.
In \cite{Shen-1} we established the sharp $O(\sqrt{\e}) $ convergence rate in a bounded domain
by constructing some boundary correctors.
In this paper we will investigate the large-scale regularity problem for solutions $(u_\e, p_\e)$.

To describe the porous domain, we let $Y=(0, 1)^d$ be an open unit cube and
$Y_s$ (solid part)  an open subset of $Y$ with Lipschitz boundary.
Throughout the paper we assume that dist$(\partial Y, \partial Y_s)>0$ and that
$Y_f= \overline{Y} \setminus \overline{Y_s}$ (the fluid part) is connected.
Let
\begin{equation}\label{omega}
\omega  =\bigcup_{z\in \mathbb{Z}^d} (Y_f +z)
\end{equation}
be the periodic repetition of $Y_f$. For $R>0$, let
\begin{equation}\label{0-1}
Q_R=(-R, R)^d \quad \text{ and } \quad Q_R^\e =Q_R \cap \e \omega.
\end{equation}

The following are the main results of the paper.

\begin{thm}\label{main-thm-1}
Let $(u_\e, p_\e)\in H^1(Q_R^\e; \R^d) \times L^2(Q_R^\e)$ be a weak solution of
\begin{equation}\label{D-1}
\left\{
\aligned
-\e^2 \mu \Delta u_\e +\nabla p_\e &  =f & \quad & \text{ in } Q_R^\e,\\
\text{\rm div} (u_\e) & =0 & \quad & \text{ in } Q_R^\e,\\
u_\e & =0 & \quad & \text{ on } Q_R \cap \partial (\e \omega),
\endaligned
\right.
\end{equation}
where $0< \e<R/2$ and $f\in C^\alpha (Q_R; \R^d)$ for some $\alpha \in (0, 1)$.
Then
\begin{equation}\label{main-1}
\aligned
 & \e \left(\fint_{Q_r} |\nabla u_\e|^2 \right)^{1/2}
+\left(\fint_{Q_r} |u_\e|^2 \right)^{1/2}
& \le C \left\{ 
\left(\fint_{Q_R} |u_\e|^2 \right)^{1/2}
 + R^\alpha \| f\|_{C^{0, \alpha} (Q_R)} \right\}
\endaligned
\end{equation}
for any $\e\le  r< R/2$, where $C$ depends only on $d$, $\mu$, $\alpha$, and $Y_s$.
\end{thm}

In \eqref{main-1} (and thereafter)   we have extended $u_\e$ to $Q_R$ by zero.
In the next theorem,  $W(y)= (W_j^i (y)) $ is a 1-periodic $d\times d$ matrix-valued function, defined  by the cell problem
\eqref{W}.

\begin{thm}\label{main-thm-2}
Let $(u_\e, p_\e)$ be the same as in Theorem \ref{main-thm-1}.
Then
\begin{equation}\label{main-2}
\aligned
& \inf_{E\in \R^d} 
 \left(\fint_{Q_r} | \e \nabla u _\e - \mu^{-1}\nabla W(x/\e) E|^2\right)^{1/2}
+\inf_{E\in \R^d} 
 \left(\fint_{Q_r} |  u_\e -  \mu^{-1} W(x/\e) E|^2\right)^{1/2}\\
& \le C \left(\frac{r}{R} \right)^\beta
\left\{ \left(\fint_{Q_R} |u_\e|^2 \right)^{1/2}
+ R^\alpha \| f\|_{C^{0, \alpha}(Q_R)} \right\}
\endaligned
\end{equation}
for any $0<\e \le r < R/2$,
where $0< \beta< \alpha$ and $C$ depends only on $d$, $\mu$, $\alpha$, $\beta$, and $Y_s$.
\end{thm}

Theorems \ref{main-thm-1} and \ref{main-thm-2} give the large-scale interior  Lipschitz and $C^{1, \alpha}$ estimates
for the Stokes equations \eqref{Stokes} in a periodically perforated domain.
We also obtain the  corresponding large-scale estimates for the pressure $p_\e$.
See Section \ref{section-p}.
We remark that the large-scale estimates for $(u_\e, p_\e)$ hold under the assumption that $Y_s$ is an open set with Lipschitz boundary.
If the boundary of $Y_s$ is smooth,
we may combine the classical regularity estimates for the Stokes equations (with $\e=1$) in $Y\setminus \overline{Y_s}$ with these large-scale estimates to 
obtain regularity estimates that are uniform   in $\e>0$.
In particular, this yields
\begin{equation}\label{full-0}
\e \|\nabla u_\e\|_{L^\infty(Q_{R/2})}
+ \| u_\e\|_{L^\infty(Q_{R/2})} 
\le C \left\{ \left(\fint_{Q_R} |u_\e|^2 \right)^{1/2}
+ R^\alpha \| f\|_{C^{0, \alpha}(Q_R)}\right\}
\end{equation}
for $0< \e\le 1$ and $R>0$,
where $C$ depends only on $d$, $\mu$, $\alpha$, and $Y_s$.
See  Remark \ref{remark-small}.

Our approach to Theorems  \ref{main-thm-1} and \ref{main-thm-2} is based on a compactness method, originated in the study of regularity problems
 for nonlinear PDEs  and minimal surfaces.
 The method was  introduced in a seminal work \cite{AL-1987} by M. Avellaneda and F. Lin  to  the study of the quantitative homogenization theory
 (see \cite{Gu-Shen} for the use of the compactness method for the Stokes equations with periodic coefficients in a fixed domain).
Let $\{ (u_{\e_j}, p_{\e_j} )\}$ be a sequence of solutions of \eqref{D-1} with $R=4$ and $\e=\e_j \to 0$.
Assume that $\{ u_{\e_j} \}$ is bounded in $L^2(Q_4; \R^d)$. 
To apply the compactness method to the Stokes equations in perforated domains with the Dirichlet condition,
the key is to extract a subsequence, still denoted by $\{( u_{\e_j}, p_{\e_j})\}$, such that 
$P_{\e_j}  \to p_0$  in $L^2(Q_1)$, where $P_\e$ is a suitable extension of $p_\e$ defined by \eqref{P},  and that 
\begin{equation}\label{s-0}
u_{\e_j} - \mu^{-1} W(x/\e_j ) (f-\nabla p_0) \to 0 \quad \text{ in } L^2(Q_1; \R^d).
\end{equation}
While  the strong  convergence $P_{\e_j} $ in $L^2$ may be proved as in the classical work 
\cite{Tartar-80, Allaire-89, LA-1990, Allaire-91a, Mik-1991, Allaire-1997} 
on Darcy's law, the strong convergence for $u_\e$ in \eqref{s-0} was only known previously in the case
when the sequence $\{u_{\e_j} \}$  has  the same Dirichlet data on a fixed boundary \cite{Allaire-91a, Allaire-1997}.
One of the main technical contributions of this work is establishing  the compactness property \eqref{s-0}
for a sequence of solutions with a uniform $L^2$ bound  for $u_\e$.
This is done by first proving a boundary layer estimate,
\begin{equation}\label{b-0}
\left( \int_{Q_{1+\delta} \setminus Q_{1-\delta}} |\e \nabla u_\e|^2\, dx \right)^{1/2}
\le C \delta^{\sigma} \big\{ \| u_\e\|_{L^2(Q_4)} + \| f\|_{L^\infty(Q_4)} \big\}
\end{equation}
for $\e \le \delta< 1/2$, where $C$ and $\sigma>0$ depend only on $d$, $\mu$,  and $Y_s$.
The proof of \eqref{b-0} uses the self-improving property of the (weak) reverse H\"older inequalities as well as an energy estimate in \cite{Shen-1}
and the nontangential-maximal-function estimates in \cite{FKV} for the Stokes equations in a bounded (unperforated) Lipschitz domain.
With \eqref{b-0} at our disposal, \eqref{s-0} is proved by applying  the  two-scale convergence method.

The large-scale regularity estimates in the homogenization theory  have been studied extensively in recent years.
Besides the compactness method, there is another approach
that is  based on the convergence rate and is  effective in both the periodic and non-periodic settings 
for second-order elliptic systems with oscillating coefficients (see \cite{Armstrong-2016, Otto-Fisher, Armstrong-book, Shen-book} for references).
In a recent work \cite{Shen-1} the present author was able to establish the sharp convergence rate for the Stokes equations 
\eqref{Stokes} in a periodically perforated domain $\Omega_\e$.
However, since the results are proved by energy estimates, the bounds for solutions $u_\e$ and their divergences cannot be separated. 
As a result, the error bound in \cite{Shen-1} requires a strong condition for the normal component of $u_\e$ on the fixed boundary $\partial \Omega$,
which is difficult  to handle in the approximation scheme.

 For second-order elliptic equations and systems in perforated domains, the large-scale regularity estimates 
  may be found  in \cite{Yeh-2010, Yeh-2016-u, Yeh-2016-1, Yeh-2015,
  Chase-Russell-2017, Armstrong-2018, Shen-2020, Gloria-2021} , 
  where the Neumann type conditions are imposed on the boundaries of
  the solid obstacles.  In this case,  the effective equations are of the same type
  and the effective solutions share the same boundary data as $u_\e$ on the fixed boundary.
  To the best of the author's knowledge,  the paper \cite{Masmoudi-2004} by N. Masmoudi seems 
  to be the only one  that treats the Stokes equations with the Dirichlet condition
  on the boundaries of solid pores.
  In particular, the uniform $W^{k, p}$ estimates  for 
  the Stokes equations \eqref{Stokes} in $\e \omega$  with smooth boundary were stated in \cite[Theorems 4.1 and 4.2]{Masmoudi-2004} without proof
   (no proof  has appeared since).   
  As a consequence of Theorem \ref{main-thm-1}, we are able to provide a proof for the uniform $W^{k, p}$ estimates.
  
 \begin{thm}\label{main-thm-3}
 Assume that  $\partial Y_s$ is of $C^{1, \alpha}$ for some $0< \alpha<1$. 
 Let $F\in L^q (\R^d; \R^d)$  and $f\in L^q(\R^d, \R^{d\times d}) $  for some $1< q<\infty$.
Then there exist a  unique $u_\e\in W^{1, q}_0(\e \omega; \R^d)$ such that
 \begin{equation}\label{main-3-0}
 \left\{
 \aligned
 -\e^2 \mu \Delta u_\e  +\nabla p_\e  & = F + \e\,  \text{\rm div} (f) & \quad & \text{ in }\  \e \omega,\\
 \text{\rm div} (u_\e ) & =0 & \quad & \text{ in }\ \e \omega,\\
 u_\e & =0 & \quad  & \text{ on }\  \partial (\e\omega),
 \endaligned
 \right.
  \end{equation}
 for some  $p_\e\in L^q_{\loc} (\e \omega)$.
 Moreover,
 \begin{equation}\label{main-3-1}
 \e \| \nabla u_\e  \|_{L^q(\e \omega)}
 + \| u_\e  \|_{L^q(\e \omega)} 
 + \e^{-1} \| \nabla p_\e\|_{W^{-1, q}(\e \omega)} 
 \le C\big\{ \| F \|_{L^q(\e \omega)} +  \| f\|_{L^q(\e \omega)} \big\},
 \end{equation}
 where $C$ depends only on $d$, $\mu$, $q$, and $Y_s$.
 \end{thm}
 
 \begin{thm}\label{main-thm-4}
 Assume that  $\partial Y_s$ is of $C^{k, \alpha} $ for some $k\ge 2$ and  $0< \alpha<1$. 
 Let $F\in W^{k-2, q} (\R^d; \R^d)$  for some $1< q<\infty$.
Then there exists a  unique $u_\e\in W^{k, q}_0(\e \omega; \R^d)$ such that
 \begin{equation}\label{main-4-0}
 \left\{
 \aligned
 -\e^2 \mu \Delta u_\e  +\nabla p_\e  & = F & \quad & \text{ in }\  \e \omega,\\
 \text{\rm div} (u_\e ) & =0 & \quad & \text{ in }\ \e \omega,\\
 u_\e & =0 & \quad  & \text{ on }\  \partial (\e\omega),
 \endaligned
 \right.
  \end{equation}
 for some  $p_\e\in L^q_{\loc} (\e \omega)$.
 Moreover,
 \begin{equation}\label{main-4-1}
\sum_{\ell=0}^k  \e^\ell  \| \nabla^\ell  u_\e  \|_{L^q(\e \omega)}
+\sum_{\ell=1}^{k} \e^{\ell-2} \| \nabla^\ell  p_\e\|_{W^{-1, q}(\e\omega)}
 \le C\sum_{\ell=0}^{k-2}  \e^\ell \|\nabla^\ell  F \|_{L^q(\e \omega)},
 \end{equation}
 where $C$ depends only on $d$, $\mu$, $q$, $k$, and $Y_s$.
 \end{thm}

The paper is organized as follows.
In Section \ref{section-2} we collect  some basic facts and estimates that will be used in later sections.
In Section \ref{section-3} we prove the crucial estimate \eqref{b-0}, which is used in the proof of
a compactness result, given in Section \ref{section-4}.
The proofs of Theorems \ref{main-thm-1} and \ref{main-thm-2} are given in Section \ref{section-5},
while the corresponding large-scale estimates for the pressure are established in Section \ref{section-p}.
Finally, Theorems \ref{main-thm-3} and \ref{main-thm-4} are
 proved in Section \ref{section-W}.
 

\section{Preliminaries}\label{section-2}

Let $Y=(0, 1)^d$ and $Y_s$ (solid part) be an open subset of $Y$ with Lipschitz boundary.
Throughout the paper we assume that  dist$(\partial Y, \partial Y_s)>0$ and that (the fluid part) 
$Y_f=\overline{Y} \setminus \overline{Y_s}$ is connected.

Let $\omega $ is given by \eqref{omega}.
Note that the unbounded domain $\omega$ is connected, 1-periodic, and $\partial \omega$ is locally Lipschitz.
Also, observe that dist$(\mathbb{Z}^d, \partial \omega)>0$.
For $1\le j \le d$, let
$( W_j (y), \pi_j (y))= (W_j^1 (y), \dots, W_j^d (y), \pi_j (y) ) \in H^1_{\loc}(\omega; \R^d)
\times L^2_{\loc} (\omega)$ be the 1-periodic solution of the cell problem,
\begin{equation}\label{W}
\left\{
\aligned
-\Delta W_j + \nabla \pi_j & =e_j & \quad & \text{ in } Y\setminus \overline{Y_s},\\
\text{\rm div} (W_j) & =0 & \quad & \text{ in } Y\setminus \overline{Y_s},\\
W_j & =0 & \quad & \text{ on } \partial Y_s,
\endaligned
\right.
\end{equation}
with $\int_{Y\setminus \overline{Y_s}}  \pi_j\, dy=0$, where $e_j=(0, \dots, 1, \dots, 0)$ with $1$ in the $j^{th}$ place.
Define
\begin{equation}\label{K}
K_j^i =\int_Y W_j^i (y)\, dy,
\end{equation}
where we have extended $W_j$ to $\R^d$ by zero.
The $d\times d$ matrix $K =(K_j^i)$, called the permeability matrix,  is symmetric and positive definite.
This follows readily from the observation
\begin{equation}\label{K-1}
K_j^i =\int_Y \nabla W_j^\ell \cdot \nabla W_i^\ell \, dy
\end{equation}
(the index $\ell$ is summed from $1$ to $d$).

Recall that $Q_R=(-R, R)^d$ and $Q_R^\e = Q_R\cap \e \omega$.

\begin{lemma}
Let $u \in  W^{1, q}(Q_R^\e)$ for some $R\in \e \N$ and $1\le q< \infty$.  
Assume $u=0$ on $Q_R \cap \partial(\e\omega)$.
Then
\begin{equation}\label{Poincare}
\| u \|_{L^q(Q_R^\e)} \le C \e \| \nabla u\|_{L^q(Q_R^\e)},
\end{equation}
where $C$ depends only on $d$, $q$,  and $Y_s$.
\end{lemma}

\begin{proof}
By dilation we may assume $\e=1$.
The result then  follows by covering $Q_R^1$ with unit cubes and applying Poincar\'e's inequality on each cube.
\end{proof}

Suppose 
$$
\left\{
\aligned
 - \e^2 \mu \Delta u_\e +\nabla p_\e  & =f \\
\text{\rm div} (u_\e) & =0 
\endaligned
\qquad \text{ in  }  Q_R^\e,
\right.
$$
with $u_\e=0$ in $Q_R\cap \partial (\e\omega)$. Let 
$$
v(x)=u_\e ( r x ), \quad q (x)=  r^{-1}  p_\e(rx ), \quad \text{ and } g(x)=f(rx),
$$
then
$$
\left\{
\aligned
  - ( \e/r )^2 \mu  \Delta v +\nabla q  & =g \\
\text{\rm div} (v) & =0 
\endaligned
\qquad \text{ in } Q_{R/r}^{\e/r},
\right.
$$
with $v=0$ in $Q_{R/r} \cap \partial ((\e/r) \omega)$.
This rescaling property will be used frequently in the paper.

\begin{lemma}
Let $(u_\e, p_\e)$ be a weak solution of  \eqref{Stokes} in $Q_R^\e $ with 
$u_\e=0$ in $Q_R\cap \partial (\e \omega)$, where $0< \e\le 1$ and  $R\in \e \N$.
Then
\begin{equation}\label{p-estimate}
\Big\| p_\e -\fint_{Q_R^\e} p_\e\Big \|_{L^2(Q_R^\e)}
\le C  R \big\{ \e \| \nabla u_\e \|_{L^2(Q_R^\e)} +  \| f\|_{L^2(Q^\e_R)}\big\},
\end{equation}
where $C$ depends only on $d$, $\mu$, and $Y_s$.
\end{lemma}

\begin{proof}
By rescaling  we may assume $R=1$.
Without loss of generality we may also  assume that $\int_{Q_1^\e} p_\e\, dx =0$.
Choose $v_\e\in H^1_0(Q_1^\e; \R^d)$ such that
$$
\text{\rm div} (v_\e)= p_\e \quad \text{ in } Q_1^\e,
$$
and 
$$
  \| v_\e \|_{L^2(Q_1^\e)} + \e  \| \nabla v_\e\|_{L^2(Q_1^\e)} \le C \| p_\e \|_{L^2(Q_1^\e)},
$$
where $C$ depends only on $d$, $\mu$, and $Y_s$.
By using $v_\e$ as a test function,  we see that
$$
\e^2\mu  \int_{Q_1^\e} \nabla u_\e \cdot \nabla v_\e\, dx
-\int_{Q_1^\e}  |p_\e|^2\, dx =\int_{Q_1^\e} f \cdot v_\e \, dx.
$$
Hence, by the Cauchy inequality,
$$
\aligned
\int_{Q_1^\e} |p_\e|^2\, dx
 & \le \e^2 \mu  \|\nabla u_\e\|_{L^2(Q_1^\e)} \|\nabla v_\e\|_{L^2(Q_1^\e)} + \| f\|_{L^2(Q^\e_1)} \| v_\e\|_{L^2(Q_1^\e)}\\
& \le C   \| p_\e\|_{L^2(Q_1^\e)}
\big\{ \e \|\nabla u_\e\|_{L^2(Q_1^\e)} +  \| f\|_{L^2(Q^\e_1)} \big\},
\endaligned
$$
which  yields  \eqref{p-estimate}.
\end{proof}

\begin{remark}\label{remark-P}
{\rm
Let $(u_\e, p_\e)$ be a weak solution of \eqref{Stokes} in $Q_R^\e$.
We extend $u_\e$ to $Q_R$ by zero and denote the extension still by $u_\e$.
For the pressure $p_\e$, we use  $P_\e$ to denote its extension defined by 
\begin{equation}\label{P}
P_\e (x)=
\left\{
\aligned
 &p_\e (x) & \quad & \text{ if } x\in Q_R^\e, \\
 & \fint_{\e (Y_f+z_k)} p_\e
 & \quad & \text{ if } x\in \e (Y_s +z_k) \text{ and } \e (Y + z_k) \subset Q_R \text{ for some } z_k \in \mathbb{Z}^d.
 \endaligned
 \right.
 \end{equation}
 See  \cite{Tartar-80, LA-1990, Allaire-1997}.
 Note that if $\e (Y +z_k)\subset Q_R$ for some $z_k \in \mathbb{Z}^d$, then 
 $$
 \fint_{\e (Y+z_k)} P_\e = \fint_{\e (Y_f +z_k)} p_\e.
 $$
 It follows that if $R \in \e \N$,
 \begin{equation}\label{P-1}
 \fint_{Q_R} P_\e =\fint_{Q_R^\e} p_\e.
 \end{equation}
}
\end{remark}

The next lemma provides a Caccioppoli type inequality for \eqref{Stokes} in perforated domains.

\begin{lemma}\label{Ca-lemma}
Let $(u_\e, p_\e)$ be a weak solution of \eqref{Stokes} 
in  $Q_{R+\e}^\e$ with $u_\e=0$ on $Q_{R+\e}\cap \partial (\e\omega)$, where $0< \e\le 1$ and $R\in \e \N$.
Then
\begin{equation}\label{Ca}
\e^2 \int_{Q_R^\e} |\nabla u_\e|^2\, dx 
+ R^{-2} \int_{Q_R^\e} |p_\e -\fint_{Q_R^\e} p_\e|^2\, dx
\le C \int_{Q^\e_{R+\e}} |u_\e|^2\, dx + C \int_{Q^\e_{R+\e}} |f|^2\, dx,
\end{equation}
where $C$ depends only on $d$, $\mu$, and $Y_s$.
\end{lemma}

\begin{proof}
In view of \eqref{p-estimate}, it suffices to bound the first term in the left-hand side of \eqref{Ca}.
By rescaling we may assume $\e=1$.
Now suppose that  
$$
-\mu \Delta u +\nabla p =f  \quad \text{ and } \quad \text{\rm div}(u) =0
$$
in $Q_{R+1}\cap \omega$, and $u=0$ in $ Q_{R+1}  \cap \partial\omega$ for some $R\in \N$.
Since dist$(\partial Y, \partial Y_s)>0$,
we may choose $\delta\in (0, 1/2) $ so small that 
$$
\widetilde{Y_f}:=(1+\delta) Y\setminus \overline{Y_s} \subset \omega.
$$
It follows from the standard Caccioppoli inequality for the Stokes equations \cite{Gia}  that 
$$
\int_{Y_f+z} |\nabla u|^2\, dx \le C \int_{\widetilde{Y_f} +z} |u|^2\, dx 
+ C \int_{\widetilde{Y_f} +z } |f|^2\, dx, 
$$
where $z\in \mathbb{Z}^d$ and $Y+z\subset Q_R$.
By summing the inequality above over $z$ we obtain  \eqref{Ca}
with $\e=1$.
\end{proof}

\begin{remark}\label{remark-Ca}
{\rm 
Let $(u_\e, p_\e)$ be a weak solution of \eqref{Stokes} in $Q_{2R}^\e$ with $u_\e=0$
on $Q_{2R}^\e \cap \partial (\e \omega)$, where $0< \e\le 1$ and $R\ge 2\e$.
Then
\begin{equation}\label{Ca-1}
\e^2 \int_{Q_R^\e} |\nabla u_\e|^2\, dx 
+ R^{-2} \int_{Q_{R}^\e} |p_\e -\fint_{Q_R^\e} p_\e|^2\, dx
\le C \int_{Q^\e_{2R}} |u_\e|^2\, dx + C \int_{Q^\e_{2R}} |f|^2\, dx.
\end{equation}
To see this, we choose $k \in \N$ such that $R\le k \e \le R + \e$.
The left-hand side of \eqref{Ca-1} is bounded by 
$$
\ \e^2 \int_{Q_{k\e} ^\e} |\nabla u_\e|^2\, dx 
+ CR^{-2}  \int_{Q_{k\e }^\e} |p_\e -\fint_{Q_{k \e} ^\e} p_\e|^2\, dx,
$$
which is bounded by the right-hand side of \eqref{Ca-1}, using \eqref{Ca} and the fact $R\ge 2\e$.
}
\end{remark}


\section{Reverse H\"older inequalities}\label{section-3}

Let   $Q(x, r)= x+ (-r, r)^d=x+ Q_r$ and $Q^\e (x, r)=Q(x, r)\cap \e \omega$. Define
\begin{equation}\label{g-e}
g_\e (x) =  \left( \fint_{Q (x, \e) } ( \e |\nabla u_\e| + |u_\e|) ^2 \right)^{1/2}.
\end{equation}
The goal of this section is to establish the following.

\begin{thm}\label{RH-thm}
Let $(u_\e, p_\e)\in H^1(Q_{2R}^\e; \R^d) \times L^2(Q_{2R}^\e)$ be a weak solution of \eqref{Stokes} in $Q^\e_{2R}$ with $u_\e=0$ on $Q_{2R} \cap \partial (\e \omega)$,
where $0< \e\le 1$ and $R \ge \e$.
Let $g_\e$ be defined by \eqref{g-e}.
Then 
\begin{equation}\label{RH-1}
\left(\fint_{Q_R} |g_\e|^q\right)^{1/q}
\le C \left(\fint_{Q_{2R} } (\e |\nabla u_\e| + |u_\e|)^2 \right)^{1/2}
+ C \left( \fint_{Q_{2R}} |f|^q \right)^{1/q},
\end{equation}
where $q>2$ and $C>0$ depend only on $d$, $\mu$, and $Y_s$.
\end{thm}

We begin with an estimate for the Stokes equations in $Q_t =(-t, t)^d$.

\begin{lemma}\label{lemma-H-1}
Let $(v, \tau)\in H^1(Q_t; \R^d) \times L^2(Q_t)$ be a weak solution of the Dirichlet problem,
\begin{equation}\label{S-1}
\left\{
\aligned
-\Delta v +\nabla \tau & =0 & \quad & \text{ in } Q_t,\\
\text{\rm div} (v)& =0 & \quad & \text{ in } Q_t,\\
v & =h & \quad & \text{ on } \partial Q_t,
\endaligned
\right.
\end{equation}
for some $t>0$, where $h \in H^1(\partial Q_t; \R^d)$ satisfies the compatibility condition
 $
 \int_{\partial Q_t} h \cdot n \, d\sigma =0.
 $
 Then there exist $q_0\in (1, 2)$ and $C>0$, depending only on $d$, such that
\begin{equation}\label{S-2}
\left(\fint_{Q_t} |v|^2 \right)^{1/2}
\le C \left(\fint_{\partial Q_t } |h|^{q_0}  \right)^{1/q_0},
\end{equation}
and
\begin{equation}\label{S-3}
\left(\fint_{Q_t} |\nabla v |^2 \right)^{1/2}
\le C \left(\fint_{\partial Q_t } |\nabla_{\tan} h|^{q_0}  \right)^{1/q_0}.
\end{equation}
\end{lemma}

\begin{proof}
By dilation we may assume $t=1$.
To prove \eqref{S-3}, we use the energy estimates to obtain 
$$
\| \nabla v \|_{L^2(Q_1)}
\le C \| h\|_{H^{1/2}(\partial Q_1)} \le C \| h \|_{W^{1, q_0}(\partial Q_1)},
$$
where $\frac{2 (d-1)}{d} < q_0<2$, and we have used the Sobolev imbedding on $\partial Q_1$ for the last inequality.
Replacing $v$ be $v-E$, with $E=\fint_{\partial Q_1} h$, we obtain \eqref{S-3} by a
Poincar\'e  inequality on $\partial Q_1$.

To see \eqref{S-2}, we use the nontangential-maximal-function estimate,
\begin{equation}\label{max}
\| (v)^*\|_{L^{q_0} (\partial Q_1)} \le C \| h \|_{L^{q_0}  (\partial Q_1)}.
\end{equation}
The estimate \eqref{max} was proved in \cite{FKV} for the Stokes equations in bounded Lipschitz domains $\Omega$, where  
$|q_0-2|<\sigma$ and $\sigma>0$ depends only on $d$ and the Lipschitz characters of $\Omega$.
As a result, \eqref{max} holds for some $\frac{2(d-1)}{d} < q_0<2$, depending only on $d$.
This, together with the estimate,
\begin{equation}\label{max-1}
\| v\|_{L^2(Q_1)} \le C \| (v)^*\|_{L^{q_0} (\partial Q_1)},
\end{equation}
gives \eqref{S-2}.

Finally, to see \ref{max-1}, we use the observation 
$$
|v(x)| \le C \int_{\partial Q_1} \frac{ (v)^*(y)}{|x-y|^{d-1}} \, d\sigma (y)
$$
for any $x\in Q_1$. It follows that
$$
\Big| \int_{Q_1} v(x) g(x)\, dx \Big|
\le C \int_{\partial Q_1}  (v)^* (y) G(y) \, d\sigma (y),
$$
where
$$
G(y) =\int_{Q_1} \frac{|g(x)|}{|x-y|^{d-1}} \, dx.
$$
Since
$$
\| G\|_{L^{q_0^\prime}(\partial Q_1)}
\le C \| G\|_{H^{1/2}(\partial Q_1)}
\le C \| G\|_{H^1(Q_1)}
\le C \| g \|_{L^2(Q_1)},
$$
we obtain \eqref{max-1} by a duality argument.
\end{proof}

In the proof of the next lemma, we will use the following observation: there exists $c_0>0$, depending only on $d$ and
$Y_s$, such that 
\begin{equation}\label{o-1}
\text{\rm dist} (\partial Q_t, \R^d \setminus \e \omega)\ge   c_0\e  \quad \text{ if }\ \ 
\text{\rm dist} (t, \e \N)\le  c_0 \e.
\end{equation}
The case $\e=1$ follows from the assumption that dist$(\partial Y, \partial Y_s)>0$, while the general case follows by dilation.

\begin{lemma}\label{lemma-H-2}
Let $(u_\e, p_\e)\in H^1(Q_{2R}^\e; \R^d) \times L^2(Q_{2R}^\e)$ be a weak solution of \eqref{Stokes} in $Q_{2R}^\e$ with
$u_\e=0$ in $Q_{2R}\cap \partial (\e \omega)$, where $0< \e\le 1$ and $R\in \e \N$.
Then
\begin{equation}\label{RH-3}
\aligned
 & \e \left(\fint_{Q_R} |\nabla u_\e|^2 \right)^{1/2}
+ \left(\fint_{Q_R} |u_\e|^2\right)^{1/2}\\
& \le C \e \left(\fint_{Q_{2R}} |\nabla u_\e|^{q_0}  \right)^{1/q_0}
+ C \left(\fint_{Q_{2R}} |u_\e|^{q_0} \right)^{1/q_0} + C \left(\fint_{Q_{2R}} |f|^2 \right)^{1/2},
\endaligned
\end{equation}
where $q_0\in (1, 2)$ is given by Lemma \ref{lemma-H-1}, and $C$ depends only on $d$, $\mu$,  and $Y_s$.
\end{lemma}

\begin{proof}
By dilation we may assume $R=1$ and $\e^{-1} \in \N$.
We first observe that by Fubini's Theorem, there exists $t\in [1, 2] $ such that
dist$(t, \e\N)\le c_0 \e$ and
\begin{equation}\label{H-21}
\e^{q_0}  \int_{\partial Q_t} |\nabla u_\e|^{q_0}\, d\sigma 
+ \int_{\partial Q_t } |u_\e|^{q_0} \, d\sigma 
\le C_0 \left\{ \e^{q_0}  \int_{Q_2} |\nabla u_\e|^{q_0}\, dx
+   \int_{ Q_2} |u_\e|^{q_0} \, dx  \right\} ,
\end{equation}
where $C_0$ depends on $d$ and $Y_s$.
For otherwise, suppose that for any $t\in [1, 2]$ with dist$(t, \e\N)\le c_0 \e$, 
$$
\e^{q_0}  \int_{\partial Q_t} |\nabla u_\e|^{q_0}\, d\sigma 
+ \int_{\partial Q_t } |u_\e|^{q_0}\, d\sigma 
> C_0\left\{  \e^{q_0}  \int_{Q_2} |\nabla u_\e|^{q_0}\, dx
+   \int_{ Q_2} |u_\e|^{q_0}\, dx  \right\} .
$$
By integrating the inequality above with respect to $t$ over the set
$$
E_\e  = \big\{ t \in (1, 2): \text{\rm dist} (t, \e \N)\le  c_0\e \big\} ,
$$
and using  the observation that $|E_\e|\ge c>0$,
 we obtain 
$$
\e^{q_0}  \int_{Q_2\setminus Q_1} |\nabla u_\e|^{q_0}\, dx
+ \int_{ Q_2\setminus Q_1 } |u_\e|^{q_0}\, dx
\ge  C_1C_0 \left\{  \e^{q_0}  \int_{Q_2} |\nabla u_\e|^{q_0} \, dx
+  \int_{ Q_2} |u_\e|^{q_0} \, dx  \right\} , 
$$
where $C_1$ depends only on $d$ and $c_0$. This gives a contradiction if we choose $C_0=(2C_1)^{-1}$.

Next, let $(v, \tau )$ be a weak solution of \eqref{S-1}  in $Q_t$ with Dirichlet data $h= u_\e$ on $\partial Q_t$.
Since dist$(\partial Q_t, \R^d\setminus \e \omega)\ge  c_0\e$,
by  the energy estimates for the Stokes equations in periodically perforated domains 
in \cite[Inequality (3.9)]{Shen-1}, we deduce that
$$
\e^2 \int_{Q_t } |\nabla u_\e|^2\, dx
+\int_{Q_t} |u_\e|^2\, dx
\le C\left\{  \e^2 \int_{Q_t} |\nabla v|^2\, dx
+\int_{Q_t} |v|^2\, dx + \int_{Q_t} |f|^2\, dx\right\}. 
$$
This, together with \eqref{S-2} and \eqref{H-21}, gives
$$
\aligned
\e \| \nabla u_\e \|_{L^2(Q_1)}
+ \| u_\e \|_{L^2(Q_1)}
& \le C \Big\{ \e \| \nabla v  \|_{L^2(Q_t)}
+ \| v  \|_{L^2(Q_t)} + \| f\|_{L^2(Q_t)}  \Big\}\\
&\le C \Big\{ \e \|\nabla_{\tan} u_\e \|_{L^{q_0}(\partial Q_t)}
+ \| u_\e\|_{L^{q_0} (\partial Q_t)} +  \| f\|_{L^2(Q_t)}  \Big\}\\
& \le C \Big\{
\e \| \nabla u_\e \|_{L^{q_0} (Q_2)} + \| u_\e\|_{L^{q_0} (Q_2)}+  \| f\|_{L^2(Q_2)}  \Big\},
\endaligned
$$
which completes the proof.
\end{proof}

\begin{remark}
{\rm 
Let $(u_\e, p_\e)$ be a weak solution of \eqref{Stokes} in $Q^\e(x_0, 4R)$
with $u_\e=0$ in $Q(x_0, 4R)\cap \partial(\e \omega)$,
where $x_0\in \R^d$, $0<\e \le 1$ and $ R\ge   2\e $.
Then 
\begin{equation}\label{RH-3-0}
\aligned
 & \e \left(\fint_{Q(x_0, R)} |\nabla u_\e|^2 \right)^{1/2}
+ \left(\fint_{Q (x_0, R)} |u_\e|^2\right)^{1/2}\\
& \le C \e \left(\fint_{Q(x_0, 4R)} |\nabla u_\e|^{q_0}  \right)^{1/q_0}
+ C \left(\fint_{Q(x_0, 4R) } |u_\e|^{q_0} \right)^{1/q_0} + C \left(\fint_{Q(x_0, 4R) } |f|^2 \right)^{1/2},
\endaligned
\end{equation}
where $q_0\in (1, 2)$ is given by Lemma \ref{lemma-H-2}.
Indeed,  by \eqref{RH-3} and translation, \eqref{RH-3-0} holds if $x_0\in \e\mathbb{Z}^d$ and
$R\in \e\N$.
Moreover, in this case, $Q (x_0, 4R)$ in the right-hand side is replaced by $Q(x_0, 2R)$.
For the general case, we choose $y_0\in \e\mathbb{Z}^d$ and $R_1\in \e \N$ such that
$$
Q (x_0, R) \subset Q (y_0, R_1) \quad 
\text{ and } \quad Q (y_0, 2R_1) \subset Q (x_0, 4R),
$$
which  is possible under the assumption $R\ge 2\e$.
}
\end{remark}

\begin{proof}[\bf Proof of Theorem \ref{RH-thm}]

By rescaling  we may assume $R=1$ and $0< \e\le 1$. 
We  also assume $0<\e< c$, where $c>0$ is sufficiently small; the case $c\le \e\le 1$ is trivial.

Let $q_0\in (1, 2)$ be given by Lemma \ref{lemma-H-2}. Define
\begin{equation}\label{G-e}
G_\e (y) =\sup
\left(\fint_{Q (z, r)} \big( \e |\nabla u_\e|+|u_\e|\big)^{q_0} \right)^{1/q_0},
\end{equation}
where the supremum is taken over all $Q (z, r)$ with the properties that
$y\in Q (z, r)$, $r\ge 2\e$, and $Q(z, r)\subset Q_2$.
We will show that
\begin{equation}\label{RH-41}
\left(\fint_{Q_1} |G_\e|^q\right)^{1/q}
\le C \left(\fint_{Q_2} |G_\e|^2 \right)^{1/2}
+ C \left(\fint_{Q_2} |f|^q\right)^{1/q}
\end{equation}
for some $q>2$, depending only on $d$, $\mu$,  and $Y_s$.
Note that by the $L^{2/q_0}$ boundedness of the Hardy-Littlewood maximal operator,
$$
\left(\fint_{Q_2} |G_\e|^2\right)^{1/2}
\le C \left(\fint_{Q_2} (\e |\nabla u_\e| +| u_\e|)^2 \right)^{1/2}.
$$
Also, observe that by \eqref{RH-3-0},
$$
\left(\fint_{Q (x, 2\e)} (\e |\nabla u_\e| +|u_\e|)^2 \right)^{1/2}
\le C G_\e (x)
+ C \left(\fint_{Q(x, 8\e)} |f|^2\right)^{1/2}
$$
for $x\in Q_1$.
It follows that
$$
\left(\fint_{Q_1} |g_\e|^q\right)^{1/q}
\le C \left(\fint_{Q_1} |G_\e|^q \right)^{1/q}
+ C \left(\fint_{Q_2} |f|^q\right)^{1/q}.
$$
As a result, the estimate \eqref{RH-1} follows from \eqref{RH-41}.

Finally, to prove \eqref{RH-41},
we use the well-known self-improving property of (weak) reverse H\"older inequalities \cite{Gia-book}. 
Consequently, 
it suffices to show that
\begin{equation}\label{RH-10}
\left(\fint_{Q(x, t)} |G_\e|^2 \right)^{1/2}
\le C \left(\fint_{Q(x, 8t)} |G_\e |^{q_0}\right)^{1/q_0}
+ C \left(\fint_{Q(x, 8t)} |f|^2 \right)^{1/2}
\end{equation}
for any $x\in Q_1$ and $0< t< c$. 
We divide the proof of \eqref{RH-10} into two cases.

Case 1. Suppose $0< t< 4 \e$. Observe that 
$$
G_\e (y)\sim G_\e (z) \quad \text{ for }  y, z\in Q(x, t).
$$
This implies that
$$
\left(\fint_{Q(x, t)} |G_\e|^2 \right)^{1/2}
\le C \left(\fint_{Q(x, 8t)} |G_\e |^{q_0}\right)^{1/q_0}.
$$

Case 2.  Suppose $4\e\le t< c$. For $y\in Q(x, t)$, write 
$$
G_\e (y)=\max \left( G^{(1)} _\e (y), G_\e^{(2)} (y)  \right),
$$
where $G_\e^{(1)}$ is defined as in \eqref{G-e}, but with the supremum being taken over all 
$Q^\e (z, r)$ with the properties that 
$y\in Q(z, r)$, $r\ge 2\e$, and $Q(z,r)\subset Q(x, 2t)$.
By the $L^{2/q_0}$ boundedness of the Hardy-Littlewood maximal operator,
 we have
 $$
 \aligned
 \left(\fint_{Q(x, t)} |G^{(1)} _\e|^2 \right)^{1/2}
 & \le C \left(\fint_{Q(x, 2t)} ( \e |\nabla u_\e| +|u_\e|)^2 \right)^{1/2}\\
& \le C \left(\fint_{Q(x, 8t)} ( \e |\nabla u_\e| +|u_\e|)^{q_0}  \right)^{1/q_0}
+ C \left(\fint_{Q(x, 8t)} |f|^2\right)^{1/2}\\
&\le C \left(\fint_{Q(x, 8t)} |G_\e |^{q_0}\right)^{1/q_0}
+ C \left(\fint_{Q(x, 8t)} |f|^2\right)^{1/2},
\endaligned
$$
where we have used \eqref{RH-3-0} for the second inequality.
Since
$$
G_\e^{(2)}  (y) \sim G_\e^{(2)} (z)  \quad \text{ for } y, z \in Q(x, t),
$$
we have
$$
\aligned
\left(\fint_{Q(x, t)} |G_\e^{(2)}|^2 \right)^{1/2}
 & \le C \left(\fint_{Q(x, t)} |G_\e^{(2)} |^{q_0} \right)^{1/q_0}\\
 & \le C \left(\fint_{Q(x, t)} |G_\e |^{q_0} \right)^{1/q_0}.
\endaligned
$$
As a result, we have proved \eqref{RH-10} for Case 2.
This completes the proof.
\end{proof}

\begin{cor}\label{cor-high}
Let $(u_\e, p_\e) \in H^1(Q_3^\e; \R^d) \times L^2(Q_3^\e)$ be a weak solution of \eqref{Stokes} in $Q_3^\e$ with 
$u_\e=0$ on $Q_3\cap \partial( \e \omega)$,
where $0< \e\le 1$. Then
\begin{equation}\label{high-1}
\aligned
& \left(\int_{Q_{1+\delta} \setminus Q_{1-\delta}}
(\e |\nabla u_\e| + |u_\e|)^2\, dx \right)^{1/2}\\
& \le C\delta^\sigma \left\{  \left(\int_{Q_3} (\e |\nabla u_\e| + | u_\e|)^2 \, dx \right)^{1/2}+  \| f\|_{L^\infty(Q_3)} \right\},
\endaligned
\end{equation}
 for any  $\delta \in (\e, 1]$, where $C$ and $\sigma>0 $ depend only on $d$, $\mu$,  and $Y_s$.
\end{cor}

\begin{proof}
We may assume $\delta\le 1/4$; for otherwise the estimate  is trivial.
By Fubini's Theorem,
$$
\left(\int_{Q_{1+\delta} \setminus Q_{1-\delta}}
(\e |\nabla u_\e| + |u_\e|)^2\, dx \right)^{1/2}
\le C \left( \int_{Q_{1+\delta} \setminus Q_{1-\delta}} |g_\e|^2\, dx \right)^{1/2},
$$
where $g_\e$ is defined by \eqref{g-e} and we have used the assumption $\delta> \e$.
By H\"older's inequality, the right-hand side of the inequality above is bounded by
$$
C \delta^\sigma \left(\int_{Q_{3/2}} |g_\e|^q \, dx \right)^{1q},
$$
where $q>2$ is given by Theorem \ref{RH-thm} and $\sigma =\frac12 -\frac{1}{q}>0$.
The estimate \eqref{high-1} now follows readily from \eqref{RH-1}.
\end{proof}


\section{Compactness}\label{section-4}
 
 The goal of this section is to establish the compactness in the following theorem.
 
\begin{thm}\label{C-thm}
Let $\{ (u_{\e_j} , p_{\e_j})\} $ be a sequence of weak solutions of 
\begin{equation}\label{C-00}
\left\{
\aligned
-\e_j ^2 \mu  \Delta u_{\e_j} + \nabla p_{\e_j}  & =f_{\e_j} & \quad & \text{ in } Q_4^{\e_j},\\
\text{\rm div} (u_{\e_j}) & =0 &\quad & \text{ in } Q_4^{\e_j},\\
u_{\e_j} &=0& \quad & \text{ on } Q_4\cap \partial (\e_j \omega),
\endaligned
\right.
\end{equation}
where $\e_j^{-1} \in \N$ and  $\e_j \to 0$.
Assume that 
\begin{equation}\label{C-1}
\fint_{Q_4} |u_{\e_j} |^2\le 1 \quad \text{ and } \quad  \| f_{\e_j} \|_{C^\alpha (Q_4)} \le 1
\end{equation}
for some $\alpha\in (0, 1)$.
Then there exists a subsequence, still denoted by $\{ (u_{\e_j}, p_{\e_j})\}$, and 
$f\in C^\alpha(Q_4; \R^d)$, $p_0 \in H^1(Q_2)$, such that
$f_{\e_j } \to f$  uniformly in $Q_4$, 
\begin{equation}\label{C-3}
P_{\e_j} -\fint_{Q_2} P_{\e_j}   \to p_0 \quad \text{ in } L^2(Q_2),
\end{equation}
\begin{equation}\label{C-2}
u_{\e_j} -\mu^{-1}  W(x/\e_j) (f-\nabla p_0)\to 0 \quad \text{ in } L^2(Q_1; \R^d),
\end{equation}
and
\begin{equation}\label{C-2-0}
\e_j \nabla u_{\e_j} -\mu^{-1}  \nabla W(x/\e_j) (f-\nabla p_0)\to 0 \quad \text{ in } L^2(Q_1; \R^{d\times d} ),
\end{equation}
where $P_{\e_j} $ denotes the extension of $p_{\e_j}$ defined  by \eqref{P}.
\end{thm}

\begin{proof}
We divide the proof of Theorem  \ref{C-thm} into several steps.

\medskip

\noindent{\bf Step 1.}
By subtracting a constant we may assume $\int_{Q_2^{\e_j}} p_{\e_j} \, dx =0$.
It follows from Caccioppoli's inequality \eqref{Ca-1} and \eqref{C-1} that
\begin{equation}\label{C-6}
\e_j  \| \nabla u_{\e_j}  \|_{L^2(Q_2)} + \| P_{\e_j}  -\fint_{Q_2} P_{\e_j}  \|_{L^2(Q_2)} \le C.
\end{equation}
Thus, by passing to a subsequence, we may assume that
\begin{equation}\label{C-7}
\left\{
\aligned
& P_{\e_j}  -\fint_{Q_2} {P_{\e_j}}  \to p_0\ \  \text{ weakly in } L^2(Q_2),\\
 & u_{\e_j}  \  \text{ two-scale converges to }  u_0 (x, \xi) ,\\
& \e_j \nabla u_{\e_j}  \  \text{ two-scale converges to }  \nabla_\xi u_0(x, \xi),
\endaligned
\right.
\end{equation}
for some $p_0\in L^2(Q_2)$ and $u_0\in L^2(Q_2; H^1_{per}(Y; \R^d))$.
Moreover, since $u_{\e_j}=0$ in $Q_2\setminus (\e_j \omega)$ and  div$(u_{\e_j} )=0$ in $Q_2^\e$,
 the limit $u_0$ satisfies
\begin{equation}\label{C-6-0}
\left\{
\aligned
& u_0(x, \xi)=0 \quad \text{ in } Q_2 \times Y_s,\\
& \text{\rm div}_\xi u_0(x, \xi) = 0 \quad \text{ in } Q_2 \times Y,\\
& \text{\rm div}_x \int_{Y} u_0 (x, \xi)\, d\xi =0 \quad \text{ in } Q_2.
\endaligned
\right.
\end{equation}
Clearly, by passing to a subsequence, we may also assume that $f_{\e_j} \to f$ uniformly in $Q_4$
  for some $f\in C^\alpha (Q_4; \R^d)$
with $\| f\|_{C^\alpha(Q_4)} \le 1$.

\medskip

\noindent{\bf Step 2.}
We show that
\begin{equation}\label{C-9}
 P_{\e_j}  -\fint_{Q_2} {P_{\e_j}}  \to p_0 \quad \text{ in } L^2(Q_2).
\end{equation}
The proof is the same as in the case with boundary value $u_\e=0$ on $\partial Q_2$.
See e.g. \cite{Allaire-1997}. We sketch a proof here for the reader's convenience.
The key is to show that for any $\psi \in H_0^1(Q_2; \R^d)$,
\begin{equation}\label{C-8}
\aligned
& |< \nabla P_{\e_j} , \psi > _{H^{-1}(Q_2)\times H_0^1(Q_2)}|\\
& \le C \Big\{ \e_j  \|\nabla u_{\e_j} \|_{L^2(Q^{\e_j} _2)} + \| f\|_{L^2(Q^{\e_j} _2)} \Big\}
\Big\{ \e_j \|\nabla \psi \|_{L^2(Q_2)}
+ \|\psi \|_{L^2(Q_2)} \Big\}.
\endaligned
\end{equation}
To see \eqref{C-8}, let  $R_{\e_j} : H_0^1(Q_2; \R^d) \to H_0^1(Q_2^{\e_j} ; \R^d)$ be the restriction  operator defined in \cite[Lemma 1.7]{Allaire-1997}.
Then
$$
\aligned
& |< \nabla P_{\e_j} , \psi > _{H^{-1}(Q_2)\times H_0^1(Q_2)}|\\
&= |< \nabla p_{\e_j} , R_{\e_j}  (\psi) > _{H^{-1}(Q_2^{\e_j} ) \times H_0^1(Q_2^{\e_j} )} |\\
&= |< \e_j^2\mu  \Delta u_{\e_j} + f_{\e_j}, R_{\e_j} (\psi) >_{H^{-1} (Q_2^{\e_j} ) \times H_0^1(Q_2^{\e_j} )} |\\
&\le \e_j^2\mu  \| \nabla u_{\e_j} \|_{L^2(Q_2^{\e_j} )}
\|\nabla R_{\e_j} (\psi)\|_{L^2(Q_2^{\e_j}) } \|
+ \| f_{\e_j}\|_{L^2(Q_2^{\e_j}) } \| R_{\e_j} (\psi) \|_{L^2(Q_2^{\e_j}) }\\
& \le C 
 \Big\{ \e_j  \|\nabla u_{\e_j} \|_{L^2(Q^{\e_j} _2)} + \| f_{\e_j} \|_{L^2(Q^{\e_j} _2)} \Big\}
\Big\{ \e_j \|\nabla \psi \|_{L^2(Q_2)}
+ \|\psi \|_{L^2(Q_2)} \Big\}.
\endaligned
$$
The estimate \eqref{C-8} implies \eqref{C-9}.
For otherwise, $\nabla P_{\e_j}$ does not converge to $\nabla p_0$ in $H^{-1}(Q_2; \R^d)$.
It follows that there exists a sequence $\{ \psi_j \}\subset  H_0^1(Q_2; \R^d)$ such that
$\|\psi_j \|_{H_0^1(Q_2)} =1$ and
$$
| < \nabla P_{\e^\prime_j} -\nabla p_0, \psi_j>_{H^{-1}(Q_2) \times H_0^1(Q_2)} |
\ge c_0>0
$$
for a subsequence $\{\e^\prime_j\}$.
By passing to a subsequence we may assume $\psi_j  \to \psi_0$ weakly in $H_0^1(Q_2; \R^d)$ and 
thus strongly in $L^2(Q_2; \R^d)$.
Since
$$
<\nabla P_{\e_j}-\nabla p_0, \psi_0>_{H^{-1}(Q_2)\times H_0^1(Q_2)} \to 0,
$$
we see that 
$$
| < \nabla P_{\e^\prime_j}, \psi_j -\psi_0>_{H^{-1}(Q_2) \times H_0^1(Q_2)} |
\ge c_0/2
$$
if $j$ is sufficiently  large. This leads to a contradiction if we take  $\psi =\psi_j -\psi_0$ in \eqref{C-8}.

\medskip

\noindent{\bf Step 3.} We show that
\begin{equation}\label{C-10}
u_0(x, \xi) = \mu^{-1} W(\xi) (f-\nabla p_0) \quad \text{ in } Q_2.
\end{equation}

 By using the Stokes equations in $Q_2^\e$
 and  the two-scale convergence of $\e_j \nabla u_{\e_j}$, we have
\begin{equation}\label{C-13}
\mu \int_{Q_2 \times Y}
\nabla_\xi u_0(x, \xi) \cdot \nabla_\xi \psi (x, \xi)\, dx d\xi
=\int_{Q_2 \times Y} f(x) \psi (x, \xi)\, dx d\xi
\end{equation}
for any $\psi = \psi (x, \xi)\in L^2(Q_2; H_{per}^1(Y; \R^d))$ satisfying the conditions,
\begin{equation}\label{C-13-0}
\left\{
\aligned
& \text{\rm div}_\xi \psi (x, \xi)=0 \quad \text{ in } Q_2 \times Y,\\
&\psi (x, \xi)=0 \quad \text{ in }  Q_2 \times Y_s,\\
&\text{\rm div}_x \int_Y \psi (x, \xi)\, d\xi =0 \quad \text{ for } x\in Q_2,\\
& n\cdot \int_{Y} \psi (x, \xi) \, d\xi =0 \quad \text{ for } x\in \partial Q_2,
\endaligned
\right.
\end{equation}
where $n$ denotes the outward unit normal to $\partial Q_2$.
See \cite[p.48-89]{Allaire-1997}.
Let $p_*\in H^1(Q_2)$ be a weak solution of the Neumann problem,
\begin{equation}\label{C-11}
\left\{
\aligned
\text{\rm div} (K (f-\nabla p_*))  & =0 & \quad & \text{ in }  Q_2,\\
\mu n \cdot K (f-\nabla p_*)  & = n\cdot \fint_Y u_0 (x, \xi) \, d\xi & \quad & \text{ on } \partial Q_2,
\endaligned
\right.
\end{equation}
and
\begin{equation}\label{C-12}
v_0(x, \xi) = \mu^{-1} W(\xi) (f-\nabla p_*) \quad \text{ in } Q_2.
\end{equation}
It is not hard to show that \eqref{C-13} also holds if $u_0(x, \xi)$ is replaced by $v_0(x, \xi)$.
Thus,
$$
\int_{Q_2\times Y}
\nabla_\xi ( u_0(x, \xi) - v_0(x, \xi)) \cdot \nabla_\xi \psi (x, \xi)\, dx d\xi=0
$$
for any $\psi = \psi (x, \xi)\in L^2(Q_2; H_{per}^1(Y; \R^d))$ satisfying \eqref{C-13-0}.
By taking $\psi =u_0-v_0$, we see that
$u_0-v_0$ depends only on $x$. Since $u_0(x, \xi)-v_0(x, \xi)=0$ for $\xi \in Y_s$,
we conclude that $u_0(x, \xi)=v_0 (x, \xi)$ in $Q_2 \times Y$.

It remains to show that $\nabla p_*=\nabla p_0$ in $Q_2$. To this end, we note that
by using the Stokes equations in $Q_2^\e$, \eqref{C-9} and the two-scale convergence of $\e_j \nabla u_{\e_j}$,
\begin{equation}\label{C-15}
\aligned
\mu \int_{Q_2\times Y} \nabla_\xi u_0 (x, \xi)  & \cdot \nabla_\xi \psi (x, \xi)\, dx d\xi
 -\int_{Q_2\times Y} p_0(x)  \, \text{\rm div}_x \psi (x, \xi)\, dx d\xi\\
& =\int_{Q_2 \times Y} f(x) \psi (x, \xi)\, dx d\xi,
\endaligned
\end{equation}
if $\psi \in C_0^\infty(Q_2; H^1_{per}(Y))$ satisfies 
$\text{\rm div}_\xi \psi (x, \xi)=0$ in $Q_2\times Y$ and
$\psi(x, \xi)=0$ in $Q_2\times Y_s$.
By taking $\psi=\varphi (x) W_\ell (\xi)$ in \eqref{C-15}, where $1\le \ell \le d$ and $\varphi \in C_0^\infty(Q_2)$, we obtain 
$$
K_\ell ^j \int_{Q_2} \big (f^j- \frac{\partial p_*}{\partial x_j} \big) \varphi\, dx
-K_\ell^j \int_{Q_2} p_0(x) \frac{\partial \varphi}{\partial x_j}\, dx
=K_\ell^j \int_{Q_2} f^j \varphi\, dx,
$$
 where we also used the fact $u_0(x, \xi) = \mu^{-1} W(\xi) (f-\nabla p_*)$.
 It follows that 
 $$
 K_\ell^j \int_{Q_2}  \varphi \frac{\partial}{\partial x_j} \big( p_* -p_0) \, dx =0
 $$
 for $1\le \ell \le d$.
 Since $K=(K_\ell^j)$ is invertible and $\varphi \in C_0^\infty (Q_2)$ is arbitrary, we deduce that
 $\nabla(p_*-p_0)=0$ in $Q_2$.
 
\medskip

\noindent{\bf Step 4.}
We show that 
\begin{equation}\label{C-14}
\e_j \nabla u_{\e_j} - \mu^{-1} \nabla W(x/\e_j) (f-\nabla p_0) \to 0 \quad \text{ in } L^2(Q_1; \R^{d\times d} ).
\end{equation}

Let
\begin{equation}\label{C-16}
I_j =\| \mu \e_j \nabla u_{\e_j} - \nabla W(x/\e_j) (f-\nabla p_0)\|^2_{L^2(Q_{1})} .
\end{equation} 
Observe that
$$
\aligned
I_j & = \e_j^2\mu^2  \int_{Q_{1}}  |\nabla u_{\e_j}|^2\, dx
 -2\mu \int_{Q_{1}} \e_j \nabla u_{\e_j} \cdot \nabla W(x/\e_j) (f-\nabla p_0)\, dx\\
&\qquad\qquad +\int_{Q_{1}} |\nabla W(x/\e_j) (f-\nabla p_0)|^2 \, dx\\
&=I_j^1 +I_j^2 +I_j^3.
\endaligned
$$
Since  $\e\nabla u_{\e_j}$ two-scale converges to $\nabla_\xi u_0 (x, \xi)=\mu^{-1} \nabla W(\xi) (f-\nabla p_0)$ in $Q_2$,  we see that 
$$
\aligned
I_j^2 +I_j^3 \to  & - \int_{Q_{1} \times Y} |\nabla W(\xi) (f-\nabla p_0)|^2\, dx d\xi \\
&=-\int_{Q_{1}} K (f-\nabla p_0) \cdot (f-\nabla p_0)\, dx\\
&=-\mu \int_{Q_1} \overline{u} \cdot (f-\nabla p_0)\, dx,
\endaligned
$$
where $\overline{u}=\mu^{-1} K (f-\nabla p_0)$.
To handle $I_j^1$, we fix $\delta \in (0, 1/8)$ and  choose a cut-off function $\varphi =\varphi_\delta \in C_0^\infty(Q_1)$ such that 
$0\le \varphi \le 1$, $\varphi (x) =0$ if dist$(x, \partial Q_1)\le \delta/2$,
$\varphi(x)=1$ if  $x\in Q_1$ and dist$(x,  \partial Q_1)\ge \delta$, and
$|\varphi|\le C \delta^{-1}$.
Note that
$$
\aligned
I_j^1 &=\mu^2 \e_j^2 \int_{Q_1^\e} |\nabla u_{\e_j}|^2 \varphi\, dx +\mu^2  \e_j^2 \int_{Q_1^\e} |\nabla u_{\e_j}|^2  (1-\varphi)\, dx\\
&=\mu \int_{Q^\e _1} ( u_{\e_j} \cdot f_{\e_j} )  \varphi \, dx +\mu  \int_{Q^\e_1} (u_{\e_j} \cdot \nabla \varphi) (P_{\e_j} -\fint_{Q_2} P_{\e_j} ) \, dx\\
& \qquad
-\mu^2 \e_j^2 \int_{Q_1^\e} u_{\e_j} (\nabla u_{\e_j}) (\nabla \varphi)\, dx+\mu^2  \e_j^2 \int_{Q_1^\e} |\nabla u_{\e_j}|^2  (1-\varphi)\, dx,
\endaligned
$$
where we have used the Stokes equations in $Q_1^\e$ and integration by parts.
By the strong convergence of $f_{\e_j}$ and $P_{\e_j}  -\fint_{Q_2} P_{\e_j}$ and
weak convergence of $u_{\e_j}$ in $L^2(Q_2)$,
it  follows that
$$
\aligned
 &\limsup_{j \to \infty}
\Big|I_j^1 -\mu \int_{Q_1} (\overline{u} \cdot f) \varphi\, dx 
-\mu \int_{Q_1} (\overline{u} \cdot \nabla \varphi) p_0\, dx\Big|\\
&\le C \sup_j \int_{Q_1^\e \setminus Q_{1-\delta} }\e_j^2  |\nabla u_{\e_j}|^2\, dx\\
&\le C \delta^{2\sigma},
\endaligned
$$
where we have used \eqref{high-1} for the last inequality.
Since
$$
\int_{Q_1} (\overline{u} \cdot f) \varphi\, dx 
+\int_{Q_1} (\overline{u} \cdot \nabla \varphi) p_0\, dx
=\int_{Q_1} (\overline{u} \cdot (f-\nabla p_0) ) \varphi\, dx,
$$
we have proved that
$$
\aligned
\limsup_{j \to \infty} |I_j|
&\le \limsup_{j \to \infty} 
\Big| I_j^1 -\mu \int_{Q_1} \overline{u} \cdot (f-\nabla p_0)\, dx \Big|\\
&\le C \delta^{2\sigma} +\mu  \int_{Q_1} |\overline{u}| |f-\nabla p_0| (1-\varphi)\, dx\\
&\le C \delta^{2\sigma} + \mu \int_{Q_1\setminus Q_{1-\delta}}  |\overline{u}| |f-\nabla p_0| \, dx, 
\endaligned
$$
where $C$ does not depend on $\delta$.
By letting  $\delta\to 0$, we conclude $I_j \to0$, as $j \to \infty$.
\medskip

\noindent{\bf Step 5.}
We show that 
\begin{equation}\label{C-50}
 u_{\e_j} -  \mu^{-1}W(x/\e_j) (f-\nabla p_0) \to 0 \quad \text{ in } L^2(Q_1; \R^d).
\end{equation}

Since $f\in C^\alpha(Q_4; \R^d)$ and $\text{\rm div}(K(f-\nabla p_0))=0$ in $Q_2$, it follows that 
$\nabla p_0\in C^\alpha(Q_1; \R^d)$.
As a result, we may choose a sequence $\{ F_j \}\subset C^1(Q_1; \R^d)$ such that 
$$
\| F_j - (f-\nabla p_0) \|_{L^\infty(Q_1)} \to 0 \quad \text{ as } j \to \infty,
$$
and $\| \nabla F_j \|_{L^\infty (Q_1)} \le C \e_j^{\alpha -1}$.
Note that
$$
\aligned
& \| u_{\e_j} - \mu^{-1} W(x/\e_j) (f-\nabla p_0)\|_{L^2(Q_1)}\\
&\le \| u_{\e_j} -\mu^{-1} W(x/\e_j) F_j\|_{L^2(Q_1)}
+ \| \mu^{-1} W(x/\e_j) (F_j - (f-\nabla p_0))\|_{L^2(Q_1)}\\
&\le C \e_j \| \nabla ( u_{\e_j} - \mu^{-1} W(x/\e_j) F_j) \|_{L^2(Q_1)}
+ \| \mu^{-1} W(x/\e_j) (F_j - (f-\nabla p_0))\|_{L^2(Q_1)}\\
&\le C  \| \e_j \nabla u_{\e_j} -\mu^{-1}\nabla W(x/\e_j) F_j \|_{L^2(Q_1)}
+ C \e_j \| W(x/\e_j) \nabla F_j \|_{L^2(Q_1)}\\
&\qquad\qquad
+ \|\mu^{-1}  W(x/\e_j) (F_j - (f-\nabla p_0))\|_{L^2(Q_1)}\\
&\le C \| \e_j \nabla u_{\e_j} - \mu^{-1}\nabla W(x/\e_j) (f-\nabla p_0) \|_{L^2(Q_1)}
+C \| \nabla W(x/\e_j) (F_j -(f-\nabla p_0 ))\|_{L^2(Q_1)}\\
& \qquad\qquad
+  C \e_j \| W(x/\e_j) \nabla F_j \|_{L^2(Q_1)}
+  C \| W(x/\e_j) (F_j - (f-\nabla p_0))\|_{L^2(Q_1)}\\
&\le  C \| \e_j \nabla u_{\e_j} - \mu^{-1} \nabla W(x/\e_j) (f-\nabla p_0) \|_{L^2(Q_1)}\\
&\qquad\qquad
+ C \| F_j - (f-\nabla p_0)\|_{L^\infty(Q_1)}
+ C \e_j \| \nabla F_j \|_{L^\infty(Q_1)},
\endaligned
$$
where we have used the Poincar\'e inequality \eqref{Poincare} for the second inequality.
As a result, \eqref{C-50} follows from \eqref{C-14}.
This completes the proof of Theorem \ref{C-thm}.
\end{proof}

\begin{remark}\label{remark-C}
{\rm 
 It follows from the proof of Theorem \ref{C-thm} that
\begin{equation}\label{C-5}
u_{\e_j} \to \overline{u}:= \mu^{-1} K(f-\nabla p_0) \quad \text{ weakly  in } L^2(Q_2; \R^d).
\end{equation}
Since $\text{\rm div} (u_{\e_j})=0$ in $Q_2$, we obtain 
\begin{equation}\label{C-4}
\text{\rm div}(K (f-\nabla p_0))=0 \quad \text{ in } Q_2.
\end{equation}
}
\end{remark}


\section{Large-scale estimates for the velocity}\label{section-5}

In this section we give the proof of Theorems \ref{main-thm-1} and \ref{main-thm-2}.

\begin{lemma}\label{lemma-L-1}
Let $0< \beta<\alpha <1$.
There exist $\theta \in (0, 1/4)$ and $\e_0\in (0, 1/4)$, depending only on 
$d$, $\alpha$, $\beta$, $\mu$, and $Y_s$,
 such that $\theta^{-1}\in 4\N$, $\e_0^{-1} \in 4 \N$, and 
\begin{equation}\label{L-1}
\aligned
 & \inf_{E\in \R^d}
\left(\fint_{Q_\theta} 
|u_\e - \mu^{-1} W(x/\e) E|^2 \right)^{1/2} \le \theta^\beta \max
\left\{ \left(\fint_{Q_1} |u_\e|^2 \right)^{1/2},
 \| f\|_{C^{0, \alpha}(Q_1)}  \right\}, 
\endaligned
\end{equation}
whenever  $0< \e<\e_0$,  $\e^{-1} \in  4\N$,  and $(u_\e, p_\e)\in H^1(Q_1^\e; \R^d) \times L^2(Q_1)$ 
is a weak solution of the Stokes equations \eqref{Stokes}  in $Q_1^\e$,
$u_\e=0$ in  $Q_1\cap \partial (\e \omega)$, and $f\in C^\alpha(Q_1; \R^d)$ with $f(0)=0$.
\end{lemma}

\begin{proof}
The lemma is proved by contradiction.
We begin by choosing $\theta\in (0, 1/4)$ such that $\theta^{-1}\in 4 \N$ and
$C_0\theta^\alpha \le (1/2) \theta^\beta$, where $C_0$ is the constant in \eqref{C-0}, which depends only on $d$, $\mu$, and $Y_s$.
This is possible since $\beta<\alpha $.

Suppose that no $\e_0$ with the desired properties exists for this $\theta$.
Then there exist a sequence of weak solutions $(u_{\e_j}, p_{\e_j}) $ of the Stokes equations,
$$
\left\{ 
\aligned
- \e_j^2 \mu \Delta u_{\e_j} +\nabla p_{\e_j} & = f_{\e_j}, \\
\text{\rm div} (u_{\e_j})  &=0,
\endaligned
\right.
$$
in $Q_1^{\e_j}$ with $u_{\e_j}=0$ on $Q_1 \cap \partial (\e_j \omega)$ such that $\e_j^{-1}  \in 4\N$, 
$\e_j \to 0$, 
\begin{equation}\label{L-2}
\max \left\{ \left(\fint_{Q_1} |u_{\e_j}|^2\right)^{1/2},  \| f\|_{C^{0, \alpha}(Q_1)}\right\}  \le 1,
\end{equation}
and
\begin{equation}\label{L-3}
 \inf_{E\in \R^d}
\left(\fint_{Q_\theta} 
|u_{\e_j}  - \mu^{-1} W(x/\e_j) E|^2 \right)^{1/2}
> \theta^\beta.
\end{equation}
By subtracting a constant we may assume $\int_{Q_{1/2}^{\e_j} } p_{\e_j} \, dx=0$.
It follows that
$$
\fint_{Q_{1/2}} P_{\e_j} = \fint_{Q_{1/2}^{\e_j} } p_{\e_j} =0.
$$
In view of Theorem \ref{C-thm}, by passing to a subsequence, we may assume that
$f_{\e_j} \to f$ uniformly in $Q_1$  for some $f\in C^\alpha (Q_1; \R^d)$,
\begin{equation}\label{L-4}
P_{\e_j}   \to p_0 \quad \text{ in } L^2(Q_{1/2}),
\end{equation}
and
\begin{equation}\label{L-5}
u_{\e_j} - \mu^{-1} W(x/\e_j) (f-\nabla p_0) \to 0 \quad \text{ in } L^2(Q_{1/4}; \R^d),
\end{equation}
for some $p_0\in H^1(Q_{1/2})$.
Note that
$$
\aligned
&\left(\fint_{Q_\theta} |u_{\e_j} - \mu^{-1} W(x/\e_j) E|^2\right)^{1/2}\\
& \le \left(\fint_{Q_\theta} |u_{\e_j} - \mu^{-1} W(x/\e_j)(f-\nabla p_0)|^2 \right)^{1/2}
+ \mu^{-1} \left(\fint_{Q_\theta} |W(x/\e_j) (f-\nabla p_0 - E)|^2 \right)^{1/2}\\
&\le \left(\fint_{Q_\theta} |u_{\e_j} -\mu^{-1}  W(x/\e_j)(f-\nabla p_0)|^2 \right)^{1/2}
+ C \| f-\nabla p_0 -E\|_{L^\infty(Q_\theta)}\\
&\le \left(\fint_{Q_\theta} |u_{\e_j} - \mu^{-1} W(x/\e_j)(f-\nabla p_0)|^2 \right)^{1/2}
+ C \theta^\alpha \Big\{ \| f\|_{C^{0, \alpha}(Q_{1/4} )} + \|\nabla p_0\|_{C^{0, \alpha} (Q_{1/4} )} \Big\},
\endaligned
$$
where we have let $E=\nabla p_0 (0)$ and used the assumption $f(0)=0$.
By letting $j \to \infty$ and using \eqref{L-3} and \eqref{L-5}, we obtain 
$$
\aligned
\theta^\beta
& \le C \theta^\alpha \Big\{ \| f\|_{C^{0, \alpha}(Q_{1/4})} + \|\nabla p_0\|_{C^{0, \alpha} (Q_{1/4})} \Big\}\\
&\le C \theta^\alpha   \Big\{ \| f\|_{C^{0, \alpha}(Q_{1})} + \| p_0\|_{L^2(Q_{1/2})} \Big\},
\endaligned
$$
where, for the last step,
we have used the interior $C^{1, \alpha}$ estimates for the elliptic equation  $\text{\rm div} (K(f-\nabla p_0))=0$ in $Q_{1/2}$
(see Remark \ref{remark-C}).

Finally, by the Caccioppoli inequality  \eqref{Ca-1}, 
$$
\| p_{\e_j} \|_{L^2(Q_{1/2} ^{\e_j})} \le C.
$$
This, together with \eqref{L-4}, yields $\| p_0\|_{L^2(Q_{1/2})} \le C$.
Hence, 
\begin{equation}\label{C-0}
\theta^\beta \le C_0 \theta^\alpha,
\end{equation}
where $C_0>0$ depends only on $d$, $\mu$,  and $Y_s$.
This is a contradiction with the choice of $\theta$.
\end{proof}

\begin{remark}\label{remark-L-1}
{\rm
Note that if $v_\e = W_j (x/\e)$ and $q_\e = \e^{-1} \pi_j (x/\e) - x_j$, then 
$$
\left\{
\aligned
-\e^2 \Delta v_\e +\nabla q_\e & =0,\\
\text{\rm div} (v_\e) & =0,
\endaligned
\right.
$$
in $\R^d\setminus \e\omega$ and $v_\e =0$ on $\partial\omega$.
This allows us to replace $u_\e$ in \eqref{L-1} by $u_\e - \mu^{-1} W(x/\e) E_0$ for any $E_0\in \R^d$.
It follows that \eqref{L-1} in Lemma \ref{lemma-L-1}  may be replaced by
\begin{equation}\label{L-1-0}
\aligned
 & \inf_{E\in \R^d}
\left(\fint_{Q_\theta} 
|u_\e - \mu^{-1} W(x/\e) E|^2 \right)^{1/2} \\
& \le \theta^\beta \max
\left\{  \inf_{E\in \R^d} \left(\fint_{Q_1} |u_\e- \mu^{-1} W(x/\e) E|^2 \right)^{1/2},
 \| f\|_{C^{0, \alpha}(Q_1)}  \right\}.
\endaligned
\end{equation}
}
\end{remark}

\begin{lemma}\label{lemma-L-2}
Let $0< \beta< \alpha<1$.
Let $\theta, \e_0\in (0, 1/4)$ be given by Lemma \ref{lemma-L-1}. Then
\begin{equation}\label{L-2-1}
\inf_{E\in \R^d}
\left(\fint_{Q_{\theta^k}} 
|u_\e - \mu^{-1} W(x/\e) E|^2 \right)^{1/2}
\le 
\theta^{k\beta}
\max \left\{ \left(\fint_{Q_1} |u_\e|^2\right)^{1/2}, \| f\|_{C^{0, \alpha} (Q_1)} \right\},
\end{equation}
whenever  $0< \e<\theta^{k-1} \e_0$,  $\e^{-1} \in  4\N$,  and $(u_\e, p_\e)\in H^1(Q_1^\e; \R^d) \times L^2(Q_1)$ 
is a weak solution of the Stokes equations \eqref{Stokes}  in $Q_1^\e$,
$u_\e=0$ in  $Q_1\cap \partial (\e \omega)$, and $f\in C^\alpha(Q_1; \R^d)$ with $f(0)=0$.
\end{lemma}

\begin{proof}
The lemma is proved by induction.
The case $k=1$ is given by \eqref{L-1-0}.

Suppose the estimate \eqref{L-2-1} holds for some $k\ge 1$.
 Let $(u_\e, p_\e)\in H^1(Q_1^\e; \R^d) \times L^2(Q_1)$ 
be a weak solution of the Stokes equations \eqref{Stokes}  in $Q_1^\e$,
$u_\e=0$ in  $Q_1\cap \partial (\e \omega)$, and $f\in C^\alpha(Q_1; \R^d)$ with $f(0)=0$.
Assume that $0< \e<\theta^k \e_0$ and $\e^{-1} \in 4\N$.
Consider
$$
v(x) =u_\e (\theta^k x) \quad \text{ and } \quad q(x)= \theta^{-k} p_\e (\theta^k x).
$$
Then 
$$
\left\{
\aligned
- (\e \theta^{-k})^2 \mu \Delta v + \nabla q & =g, \\
\text{\rm div}(v)  & =0,
\endaligned
\right.
$$
in  $Q_1^{\theta^{-k} \e}$, and $v=0$ on $Q_1\cap \partial (\e \theta^{-k} \omega)$, where
$g(x) =f (\theta^k x)$.
Since $\theta^{-k} \e<\e_0$,
it follows from \eqref{L-1-0}that
$$
\aligned
 & \inf_{E\in \R^d}
\left(\fint_{Q_{\theta^{k+1} }} 
|u_\e - \mu^{-1} W(x/\e) E|^2 \right)^{1/2}
=\inf_{E\in \R^d}
\left(\fint_{Q_\theta} | v -\mu^{-1}  W(x/(\e \theta^{-k}) ) E|^2 \right)^{1/2}\\
&\le \theta  \max \left\{ 
 \inf_{E\in \R^d}\left( 
\fint_{Q_1} |v- \mu^{-1} W(x/(\e \theta^{-k}) )E |^2 \right)^{1/2}, \| g\|_{C^{0, \alpha}(Q_1)} \right\}\\
&= \theta  \max \left\{ 
 \inf_{E\in \R^d}\left( 
\fint_{Q_{\theta^k} } |u_\e- \mu^{-1} W(x/\e )E |^2 \right)^{1/2}, \theta^{k\alpha} \| f\|_{C^{0, \alpha}(Q_{\theta^k} )} \right\}\\
&\le \theta^{(k+1) \beta }  \max \left\{ 
 \inf_{E\in \R^d}\left( 
\fint_{Q_1 } |u_\e- \mu^{-1} W(x/\e )E |^2 \right)^{1/2}, \| f\|_{C^{0, \alpha}(Q_1 )} \right\},
\endaligned
$$
where we have used the induction assumption for the last inequality.
This completes the induction argument.
\end{proof}

The next theorem gives the large-scale $C^{0, \alpha}$ estimates for the Stokes equations in perforated domains.

\begin{thm}\label{C-1-a-thm}
Let $(u_\e, p_\e)$ be a weak solution of the Stokes equations in $Q_R^\e$ with
$u_\e=0$ on $Q_R\cap \partial (\e\omega)$, where $0<\e < R$ and $f\in C^{\alpha}(Q_R; \R^d)$ for some $0<\alpha<1$.
Then
\begin{equation}\label{C-1-a}
\aligned
 & \inf_{E\in \R^d} \left(\fint_{Q_r} | u_\e -\mu^{-1} W(x/\e) E|^2\right)^{1/2}\\
&\le C \left( \frac{r}{R} \right)^\beta
\left\{ 
\left(\fint_{Q_R} |u_\e|^2 \right)^{1/2}
+ R^\alpha \| f\|_{C^{0, \alpha} (Q_R)} \right\},
\endaligned
\end{equation}
for any $\e\le r< R$, where $0< \beta< \alpha$ and
$C$ depends only on $d$, $\mu$,  $\alpha$, $\beta$, and $Y_s$.
\end{thm}

\begin{proof}
Note that \eqref{C-1-a} is trivial if $cR< r<R$. Also, observe  that
$$
-\mu \e^2 \Delta u_\e +\nabla (p_\e - f(0) \cdot x) = f- f(0).
$$
We may assume  $f(0)=0$. As a result, by Lemma \ref{lemma-L-2},
\eqref{C-1-a} holds for $\e  \le r<R=1$, if $\e^{-1} \in 4 \N$.
By considering the solution $( u_\e(tx), t^{-1} p_\e(tx))$, where $1/2< t< 1$, we deduce that
\begin{equation}\label{C-1-b}
\aligned
 & \inf_{E\in \R^d} \left(\fint_{Q_{r} } | u_\e -\mu^{-1}  W(x/\e) E|^2\right)^{1/2}\\
&\le C r^\beta
\left\{ 
\left(\fint_{Q_t} |u_\e|^2 \right)^{1/2}
+  \| f\|_{C^{0, \alpha} (Q_t)} \right\},
\endaligned
\end{equation}
if $\e < r< t$ and $t\e^{-1} \in 4 \N$.
It follows that \eqref{C-1-a} holds for $\e\le r< R=1$, without the condition $\e^{-1}\in 4\N$.
By dilation this implies that \eqref{C-1-a} holds for any $\e\le r<R$.
\end{proof}

\begin{proof}[\bf Proof of Theorem \ref{main-thm-2}]
The estimate for the second term in the right-hand side of \eqref{main-2}
is contained in Theorem \ref{C-1-a-thm}.
For the first term, we apply the Caccioppli inequality \eqref{Ca-1} to
$u_\e - \mu^{-1} W(x/\e) E$ and $p_\e-    ( \e \pi (x/\e)-x)\cdot E$.
\end{proof}

The remaining of this section is devoted to the proof of Theorem \ref{main-thm-1}.

\begin{lemma}\label{lemma-L-3}
Let $0< \beta< \alpha<1$ and  $\theta, \e_0\in (0, 1/4)$ be given by Lemma \ref{lemma-L-1}.
Let $0< \e<\theta^{k-1} \e_0$,  $\e^{-1} \in  4\N$.
Suppose $(u_\e, p_\e)\in H^1(Q_1^\e; \R^d) \times L^2(Q_1)$ 
is a weak solution of the Stokes equations \eqref{Stokes}  in $Q_1^\e$,
$u_\e=0$ in  $Q_1\cap \partial (\e \omega)$, and $f\in C^\alpha(Q_1; \R^d)$ with $f(0)=0$.
Let $E(k)  \in \R^d$ be such that
\begin{equation}\label{L-3-0}
\left(\fint_{Q_{\theta^k}} |u_\e - \mu^{-1} W(x/\e) E(k) |^2 \right)^{1/2}
=
\inf_{E\in \R^d}
\left(\fint_{Q_{\theta^k}} 
|u_\e - \mu^{-1} W(x/\e) E|^2 \right)^{1/2}.
\end{equation}
Then
\begin{equation}\label{L-3-1}
|E(k)|
\le C\Big\{ \| u_\e\|_{L^2(Q_1)} + \| f\|_{C^\alpha(Q_1)} \Big\},
\end{equation}
where $C$ depends only on $d$, $\mu$,  and $Y_s$.
\end{lemma}

\begin{proof}
The proof uses the following observation,
\begin{equation}\label{L-3-2}
|E| \le C \left(\fint_{Q_r} |\mu^{-1} W(x/\e) E|^2\right)^{1/2}
\end{equation}
for any $r\ge \e$ and $E\in \R^d$, where $C$ depends only on $d$, $\mu$,  and $Y_s$.
Let $1\le \ell \le k$ and  $E(0)=0$. Then
$$
\aligned
& |E(\ell) -E(\ell-1)|
\le C \left(\fint_{Q_{\theta^\ell}} |  \mu^{-1} W(x/\e) (E(\ell) -E(\ell-1)) |^2 \right)^{1/2}\\
&\le C \left(\fint_{Q_{\theta^\ell}} |u_\e - \mu^{-1} W(x/\e) E(\ell)|^2 \right)^{1/2}
 + C \left(\fint_{Q_{\theta^\ell}} |u_\e - \mu^{-1} W(x/\e) E(\ell-1)|^2 \right)^{1/2}\\
&  \le C \left(\fint_{Q_{\theta^\ell}} |u_\e - \mu^{-1} W(x/\e) E(\ell)|^2 \right)^{1/2}
 +C \left(\fint_{Q_{\theta^{\ell-1} }} |u_\e - \mu^{-1} W(x/\e) E(\ell-1)|^2 \right)^{1/2}\\
 &\le C \theta^{ \ell \beta }
 \Big\{ \| u_\e\|_{L^2(Q_1)} + \| f\|_{C^\alpha(Q_1)} \Big\},
 \endaligned
 $$
 where we have used \eqref{L-2-1} for the last inequality.
 It follows that
 $$
 \aligned
 |E(k)| & \le \sum_{\ell=1}^k | E(\ell)-E(\ell-1)|\\
 &\le C \big\{ \| u_\e\|_{L^2(Q_1)} + \| f\|_{C^\alpha(Q_1)} \big\}.
 \endaligned
 $$
\end{proof}

\begin{thm}\label{L-infty-thm}
Let $(u_\e, p_\e)$ be a weak solution of the Stokes equations in $Q_R^\e$ with
$u_\e=0$ on $Q_R\cap \partial (\e\omega)$, where $0<\e < R$ and $f\in C^{\alpha}(Q_R; \R^d)$ for some $0<\alpha<1$.
Then
\begin{equation}\label{L-in-1}
\left(\fint_{Q_r } |u_\e|^2\right)^{1/2}
\le C
\left\{ 
\left(\fint_{Q_R} |u_\e|^2 \right)^{1/2}
+ R^\alpha \| f\|_{C^{0, \alpha} (Q_R)} \right\},
\end{equation}
for any $\e\le r< R$, where $C$ depends only on $d$, $\mu$, $\alpha$, and $Y_s$.
\end{thm}

\begin{proof}
As in the proof of Theorem \ref{C-1-a-thm}, we may assume  $f(0)=0$.
It follows from Lemmas \ref{lemma-L-2} and  \ref{lemma-L-3} that
\eqref{L-in-1} holds for $\e\le  r< R=1$, if $\e^{-1} \in 4\N$.
The extra condition $\e^{-1} \in 4\N$ may be eliminated by considering 
$(u_\e (tx), p_\e (tx))$ for $t\in (1/2, 1)$, as in the proof of Theorem \ref{C-1-a-thm}.
Finally, the general case $\e\le r< R< \infty$ follows by a dilation argument. 
\end{proof}

\begin{proof}[\bf Proof of Theorem \ref{main-thm-1}]
The estimate for the second term in the right-hand side of \eqref{main-1}
is contained in Theorem \ref{L-infty-thm}.
For the first term, we apply the Caccioppoli inequality \eqref{Ca-1}.
\end{proof}

\begin{remark}\label{remark-small}
{\rm 
The large-scale estimates in
Theorems \ref{C-1-a-thm} and \ref{L-infty-thm} hold under the assumption that
$Y_s$ is an open subset with Lipschitz boundary.
Suppose that $Y_s$ is an open set with $C^{1, \alpha}$ boundary for some $\alpha>0$.
Using the classical Lipschitz estimates for the Stokes equations in $\widetilde{Y_f}=(1+\delta) Y \setminus Y_s$ 
\cite{Galdi, Gia} and
a rescaling argument, we see that
$$
\aligned
 & \| u_\e\|_{L^\infty(\e(Y+z))}
+\e \|\nabla u_\e\|_{L^\infty (\e (Y +z))}\\
& \le C \left\{ \left(\fint_{2\e (Y+z)} |u_\e|^2 \right)^{1/2}
+ \| f- f(z)\|_{L^\infty(2\e(Y+z))} \right\},
\endaligned
$$
for any $z\in\mathbb{Z}^d$,
where $C$ depends only on $d$, $\mu$, and $Y_s$.
This, together with \eqref{L-in-1}, gives
\begin{equation}\label{full-1}
\aligned
\| u_\e\|_{L^\infty (Q_{R/2})}
+ \e \| \nabla u_\e\|_{L^\infty (Q_{R/2})}
\le C \left\{
\left(\fint_{Q_R} |u_\e|^2 \right)^{1/2}
+ R^\alpha \| f\|_{C^{0, \alpha}(Q_R)} \right\},
\endaligned
\end{equation}
where $0<\e< R/2$ and $C$ depends only on $d$, $\mu$,  $\alpha$, and $Y_s$.
}
\end{remark}


\section{Large-scale estimates for the pressure}\label{section-p}

\begin{thm}\label{p-thm}
Let $(u_\e, p_\e)$ be a weak solution of the Stokes equations  \eqref{Stokes} in $Q_R^\e$ with
$u_\e=0$ on $Q_R\cap \partial (\e\omega)$, where $0< \e < R$ and $f\in C^{ \alpha}(Q_R; \R^d)$ for some $0<\alpha<1$.
Then
\begin{equation}\label{p-0}
\aligned
& \inf_{ \substack{ E\in \R^d\\  \gamma \in \R}}
\frac{1}{r} \left(\fint_{Q_r^\e} 
|p_\e - \gamma -x \cdot f(0) -  (\e \pi (x/\e) -x) \cdot E|^2 \right)^{1/2}\\
& \le C \left(\frac{r}{R} \right)^\beta
\left\{ \left(\fint_{Q_R^\e} |u_\e|^2 \right)^{1/2}
+ R^\alpha \| f\|_{C^{0, \alpha} (Q_R)} \right\},
\endaligned
\end{equation}
for any $ \e\le r<R/2$, where $C$ depends only on $d$, $\mu$, $\alpha$, $\beta$, and $Y_s$.
\end{thm}

\begin{proof}
By rescaling  we may assume $r=1$.
We may also assume   $\e^{-1}\in  \N$.
By  the Caccioppoli inequality \eqref{Ca-1},
\begin{equation}\label{p-0-1}
\inf_{\gamma\in \R} 
\| p_\e -\gamma  \|_{L^2(Q_1^\e)}
\le C \big\{ \| u_\e\|_{L^2(Q_{2})} + \| f\|_{L^2(Q_{2})}\big\}.
\end{equation}
By applying the estimate above to the solution 
$$
v_\e = u_\e - \mu^{-1} W(x/\e) E \quad \text{ and } \quad
q_\e=p_\e - (\e \pi (x/\e)-x) \cdot E - x \cdot f(0),
$$
we obtain 
$$
\aligned
 & \inf_{\substack{ E\in \R^d\\ \gamma \in \R}}
\left(\fint_{Q_1^\e}
|p_\e - \gamma - ( \e \pi (x/\e) -x)\cdot E - x \cdot f(0)|^2 \right)^{1/2}\\
& \le C
\inf_{E\in \R^d}
\left(\fint_{Q_{2}}
|u_\e - \mu^{-1} W(x/\e) E|^2 \right)^{1/2}
+ C  \| f\|_{C^{0, \alpha}(Q_{2})} \\
&\le C \left(\frac{1}{R} \right)^\beta\left\{ 
\left(\fint_{Q_R} |u_\e|^2\right)^{1/2}
+ R^\alpha \| f\|_{C^{0, \alpha}(Q_R)} \right\},
\endaligned
$$
where we have used \eqref{C-1-a} for the last step.
\end{proof}

\begin{thm}\label{p-thm-1}
Let $(u_\e, p_\e)$ be the same as in Theorem \ref{p-thm}.
Then
\begin{equation}\label{p-1-0}
\inf_{\gamma\in \R^d}
\frac{1}{r}
\left(\fint_{Q_r^\e} |p_\e -\gamma - x\cdot f(0)|^2 \right)^{1/2}
 \le C 
\left\{ \left(\fint_{Q_R} |u_\e|^2 \right)^{1/2}
+ R^\alpha \| f\|_{C^{0, \alpha} (Q_R)} \right\},
\end{equation}
for any $ \e\le r<R/2$, where $C$ depends only on $d$, $\mu$,  $\alpha$,  and $Y_s$.
\end{thm}

\begin{proof}
As in the proof of Theorem \ref{p-thm},
we may assume that $r=1$ and $\e^{-1} \in \N$.
It follows from \eqref{p-0-1} that
$$
\inf_{\gamma \in \R}
\| p_\e - \gamma - x\cdot f(0)\|_{^2(Q_1^\e)}
\le C \big\{ \| u_\e\|_{L^2(Q_1)} +  \| f\|_{C^{0, \alpha}(Q_2)}\big\} .
$$
The desired estimate now follows readily from \eqref{L-in-1}.
\end{proof}

\begin{remark}\label{remark-p}
{\rm
Let $(u_\e, p_\e)$ be the same as in Theorem \ref{p-thm}.
It follows from \eqref{p-1-0} that 
\begin{equation}\label{p-2-0}
\frac{1}{r}
\left(\fint_{Q_r^\e}
|p_\e - \fint_{Q_r^\e} p_\e |^2\right)^{1/2}
\le C 
\left\{ \left(\fint_{Q_R} |u_\e|^2 \right)^{1/2}
+\| f\|_{L^\infty(Q_R)}
+ R^\alpha \| f\|_{C^{0, \alpha} (Q_R)} \right\},
\end{equation}
for $\e\le r< R/2$.
This implies that
\begin{equation}
\Big|
\fint_{Q_{2r}^\e} p_\e -\fint_{Q_r^\e} p_\e\Big|
\le C r
\left\{ \left(\fint_{Q_R} |u_\e|^2 \right)^{1/2}
+\| f\|_{L^\infty(Q_R)}
+ R^\alpha \| f\|_{C^{0, \alpha} (Q_R)} \right\}
\end{equation}
for $\e\le r< R/4$.
It follows that
\begin{equation}\label{full-p-1}
\Big|
\fint_{Q_{\e }^\e} p_\e -\fint_{Q_{R/2}^\e} p_\e\Big|
\le C R
\left\{ \left(\fint_{Q_R} |u_\e|^2 \right)^{1/2}
+\| f\|_{L^\infty(Q_R)}
+ R^\alpha \| f\|_{C^{0, \alpha} (Q_R)} \right\}.
\end{equation}
Suppose that $Y_s$ is an open subset of $Y$ with $C^{1, \alpha}$ boundary.
By the classical local estimates for the Stokes equations in $(1+\delta)Y\setminus \overline{Y_s} $ \cite{Galdi, Gia} and a rescaling argument,
$$
\| p_\e -\fint_{\e (Y_f+z)} p_\e \|_{L^\infty (\e (Y_f +z))}
\le C \e \left\{
\left(\fint_{2\e (Y_f+z)} |u_\e|^2 \right)^{1/2}
+ \| f\|_{L^\infty(2\e (Y_f+z))} \right\}.
$$
This, together with \eqref{full-p-1}
\begin{equation}\label{full-p-2}
\| p_\e - \fint_{Q_{R/2}^\e} p_\e \|_{L^\infty (Q_{R/2}^\e)}
\le  C R  \left\{
\left(\fint_{Q_R} |u_\e|^2 \right)^{1/2}
+ \| f\|_{L^\infty(Q_R)}
+ R^\alpha \| f\|_{C^{0, \alpha}(Q_R)} \right\},
\end{equation}
where $0<\e<R/4$ and $C$ depends only on $d$, $\mu$, $\alpha$,  and $Y_s$.
}
\end{remark}

We end this section  by establishing a Liouville property for the Stokes equations in $\omega$.

\begin{thm}\label{LP-thm}
Let $(u, p )\in H^1_{loc} (\omega; \R^d) \times L^2_{loc} (\omega)$ be a weak solution of the Stokes
equations 
$$
-\mu \Delta u +\nabla p=f \quad \text{ and } \quad \text{\rm div}(u)=0 \qquad \text{ in } \omega,
$$
with $u=0$ on $\partial\omega$, where $f$ is constant.
Assume that there exist some $C>0$ and $\sigma \in (0, 1)$ such that 
\begin{equation}
\left(\fint_{Q_R} |u|^2\right)^{1/2}
\le C R^\sigma
\end{equation}
for any $R>1$. Then
\begin{equation}
u=\mu^{-1}  W(x)E \quad \text{ and } \quad p=  (\pi (x) -x) \cdot E+ x\cdot f(0) +\gamma,
\end{equation}
for some $E\in \R^d$ and $ \gamma \in \R$.
\end{thm}

\begin{proof}
Choose $\alpha$, $\beta$ so that $\sigma < \beta< \alpha<1$.
We apply the estimate \eqref{C-1-a} with $\e=1$ to $(u, p)$ and let $R\to \infty$.
It follows that  for each $k\in \N$,
$
u = \mu^{-1} W(x)E (k)   \text{ in } Q_k\cap \omega
$
for some $E(k) \in \R^d$. 
Since
$$
\mu \int_{Y_f}   u\, dx =\int_{Y_f} W(x) E(k)\, dx = K E(k),
$$
and $K=(K_j^i)$ is invertible, we see that $E(k) =E(k+1)$ for any $k\in \N$. 
This implies that $u=\mu^{-1} W(x) E$  in $\R^d$ for some $E\in \R^d$.
It follows that
$$
\nabla \{ p-  (\pi (x) -x) \cdot E - x \cdot f(0)\} =0 \quad \text{ in } \omega.
$$
Since $\omega$ is connected, we conclude that
$p=(\pi (x) -x )\cdot E + x\cdot f(0) +\gamma$ for some $\gamma \in \R$.
\end{proof}


\section{Uniform $W^{k, p}$ estimates}\label{section-W}

In this section we give the proof of Theorems \ref{main-thm-3} and \ref{main-thm-4}.
By rescaling we may assume $\e=1$.

\begin{proof}[\bf Proof of Theorem \ref{main-thm-3}]

\noindent{\bf Step 1.}
The case  $q=2$. 

Le $V$ denote the closure of $\mathcal{V}$ in $W_0^{1, 2}(\omega; \R^d)$, where
\begin{equation}\label{V}
\mathcal{V}= \big\{ \psi \in C_0^\infty(\omega; \R^d): \ \text{\rm div}(\psi)=0 \text{ in } \omega \big\}.
\end{equation}
Using the inequality 
$
\| u\|_{L^2(\omega)} \le C \|\nabla u \|_{L^2(\omega)}
$
for any $u\in W_0^{1, 2} (\omega)$, and the Lax-Milgram Theorem, one may show that for each $F\in L^2(\R^d; \R^d)$ and
 $f\in L^2(\R^d; \R^{d\times d} )$, there
exists  a unique $u\in V$ such that
\begin{equation}\label{weak}
\mu \int_\omega \nabla u \cdot \nabla \psi \, dx =\int_\omega F \cdot \psi \, dx -\int_\omega f \cdot \nabla \psi\, dx
\end{equation}
for any $\psi \in V$. Moreover, $u$ satisfies the estimate \eqref{main-3-1} with $q=2$ and $\e=1$, and
$-\mu \Delta u +\nabla p =F+\text{div}(f)$ in $\omega$ for some
$p\in L^2_{\loc}(\omega)$.

 \medskip
 
 \noindent{\bf Step 2.}
 Let $(u, p)$ be the weak solution of \eqref{main-3-0} with $\e=1$,  given by Step 1, where  $F\in C_0^\infty(\R^d, \R^d)$ and
  $f\in C_0^\infty(\R^d; \R^{d\times d} ) $. 
  We prove the estimate \eqref{main-3-1} for $2<q<\infty$ by a real variable method. 
  
  Consider the linear operator,
  $$
  T(F, f)= u,
  $$
  where  $F\in L^2(\R^d; \R^{d} )$, $f\in L^2(\R^d; \R^{d\times d} )$,  and $u$ is the  solution of \eqref{main-3-0} with $\e=1$, given by Step 1.
  Clearly, $\| T(F, f) \|_{L^2(\R^d)} \le C \| (F, f) \|_{L^2(\R^d)}$.
  We claim that if 
  supp$(F)$, supp$(f) \subset \R^d \setminus Q(x_0, 4R)$ for some $x_0\in \R^d$ and $R>0$, then
  \begin{equation}\label{W-12}
  \| T(F, f)\|_{L^\infty(Q(x_0, R)) }
  \le C \left(\fint_{Q(x_0, 4R)} |T(F, f)|^2\right)^{1/2}.
  \end{equation}
  Indeed, since $F=0$ and $f=0$ in $Q(x_0, 4R)$, we have 
  $-\mu\Delta  u +\nabla p =0$ and div$(u)=0$ in $Q^1(x_0, 4R)$, and $u=0$ on $\partial\omega$.
  If $0< R<2$, by the classical $L^\infty$  estimates for the Stokes equations
  in $C^{1, \alpha}$ domains, we obtain
    $$
   \max_{Q(x_0, R)} |u| 
  \le C \left(\fint_{Q(x_0, 4R)} |u| ^2\right)^{1/2}.
  $$
  If $R>2$, in view of \eqref{full-1}, the inequality  above continues to hold.
  As a result, by \cite[Theorem 4.2.5]{Shen-book}, we deduce that
  $$
\| T(F, f)\|_{L^q(\R^d)}
\le C_q \| (F, f)\|_{L^q(\R^d)}
$$
for any $q>2$ and  $F\in C_0^\infty(\R^d; \R^d)$, $f\in C_0^\infty( \R^d; \R^{d\times d} )$,
where $C_q$ depends only on $d$, $\mu$, $q$, and $Y_s$.
This gives the desired estimate for $u$. To bound $\nabla u$, we use the local estimate \cite{Galdi},
\begin{equation}\label{c-1}
\int_{Y_f+z} |\nabla u|^q \, dx
\le C \left\{ \int_{\widetilde{Y_f} +z} |u|^q\, dx
+ C \int_{\widetilde{Y_f} +z} |F|^q\, dx
+ \int_{\widetilde{Y_f} +z} |f|^q\, dx \right\}
\end{equation}
for $1<q<\infty$,
where $z\in \mathbb{Z}^d$ and $\widetilde{Y_f} = (1+\delta) Y \setminus \overline{Y_s}$.
It follows from \eqref{c-1} by summing over $z\in \mathbb{Z}^d$ that
\begin{equation}\label{c-2}
\aligned
\| \nabla u\|_{L^q(\omega)}
 & \le C \big\{ \| u \|_{L^q(\omega)} + \| F \|_{L^q(\R^d)} + \| f\|_{L^q(\R^d)} \big\}\\
 & \le C \big\{  \| F \|_{L^q(\R^d)} + \| f\|_{L^q(\R^d)} \big\}.
\endaligned
\end{equation}

\medskip

\noindent{\bf Step 3.} Let $(u, p)$ be the weak solution of \eqref{main-3-0},  given by Step 1, where  $F\in C_0^\infty(\R^d; \R^d)$ and 
  $f\in C_0^\infty(\R^d; \R^{d\times d} ) $. 
  We prove the estimate \eqref{main-3-1} for $1<q< 2$ by a duality argument.
  
  Let $(v, \tau)$ be the weak solution of \eqref{main-3-0}, given by Step 1, with $G\in C_0^\infty(\R^d; \R^d)$ in the place of $F$ and $g\in C_0^\infty(\R^d; \R^{d\times d})$
  in the place of $f$. Since $u, v \in V$, by \eqref{weak},
  $$
\int_\omega F \cdot v\, dx   -\int_\omega f \cdot \nabla v\, dx =\int_\omega \nabla u\cdot \nabla v\, dx
  =\int_\omega G \cdot u\, dx - \int_\omega g\cdot \nabla u\, dx.
  $$
It follows that
$$
\aligned
\Big| \int_\omega G  \cdot  u \, dx - \int_\omega g \cdot \nabla u\, dx \Big|
 & \le  \|  F \|_{L^{q} (\R^d)} \| v \|_{L^{q^\prime} (\R^d)} + 
 \|  f\|_{L^q(\R^d)} \| \nabla v \|_{L^{q^\prime}(\R^d)} \\
& \le C\big \{   \| G \|_{L^{q^\prime}(\R^d)} + 
\| g\|_{L^{q^\prime}(\R^d)}  \big\}
\big \{ \| F \|_{L^q(\R^d)} + \| f\|_{L^{q}(\R^d)} \big \}.
\endaligned
$$
By duality we obtain $\| \nabla u\|_{L^q(\R^d)} +\| u\|_{L^q(\R^d)}  \le C \{ \| F\|_{L^q(\R^d)} + \| f\|_{L^q(\R^d)} \}$ for $1< q< 2$.

\medskip

\noindent{\bf Step 4.}
The existence of solutions $u$ in $W_0^{1, q}(\omega; \R^d)$ with the estimate \eqref{main-3-1}   for general  $F\in L^q(\R^d; \R^d)$ and
$f\in L^q(\R^d; \R^{d\times d})$ follows readily from Steps 2 and 3 by a density argument.
We note that the estimate for $\nabla p$ in $W^{-1, q}(\omega; \R^d)$ follows from the equation $\nabla p=\mu \Delta u +F +\text{\rm div}(f)$.

\medskip

\noindent{\bf Step 5.} To establish  the uniqueness of solutions in $W_0^{1, q} (\omega; \R^d)$, we assume that
$u\in W^{1, q}_0(\omega; \R^d)$ is a solution
of \eqref{main-3-0} with $\e=1$ and $F=0$, $f=0$. By local estimates for the Stokes
equations in $\widetilde{Y_f} $ (see e.g. \cite{Galdi}),
$$
\max_{Y_f +z} |u|
\le C \left(\int_{\widetilde{Y_f}+z} |u|^q\, dx \right)^{1/q},
$$
where $z\in \mathbb{Z}^d$.
Since $u\in L^q(\omega; \R^d)$,
it follows that $u$ is bounded in $\omega$.
In view of Theorem \ref{LP-thm}, we deduce that $u=\mu^{-1} W(x)E$ for some $E\in \R^d$.
This shows that $u$ is 1-periodic, and thus $u=0$ in $\omega$.
\end{proof}

\begin{proof}[\bf Proof of Theorem \ref{main-thm-4}]
The uniqueness is contained in Theorem \ref{main-thm-3}.
To establish the existence and the estimate \eqref{main-4-1} with $\e=1$,
we use the local estimate \cite{Galdi},
\begin{equation}\label{c-3}
\sum_{\ell=0}^k \int_{Y_f+z} |\nabla^\ell  u|^q \, dx
\le C \left\{ \int_{\widetilde{Y_f} +z} |u|^q\, dx
+ \sum_{\ell=0}^{k-2}  \int_{\widetilde{Y_f} +z} |\nabla^\ell F|^q\, dx
 \right\}
\end{equation}
for $1<q<\infty$, where $z\in \mathbb{Z}^d$.
This yields that
$$
\aligned
\sum_{\ell=0}^k \| \nabla^ \ell u\|_{L^q(\omega)}
 & \le C \Big\{ \| u\|_{L^q(\omega)} + \sum_{\ell=0}^{k-2} \| \nabla^\ell F\|_{L^q(\omega)} \Big\}\\
 & \le C  \sum_{\ell=0}^{k-2} \| \nabla^\ell F\|_{L^q(\omega)},
 \endaligned
 $$
 where we have used \eqref{main-3-1} to bound $\| u\|_{L^q(\omega)}$ for the last inequality.
 The estimate for $\nabla^\ell  p$ in $W^{-1, q}(\omega; \R^d)$ follows by using the
 equation $\nabla p= \mu \Delta u +F$.
\end{proof}


 \bibliographystyle{amsplain}
 
\bibliography{Darcy-1.bbl}

\providecommand{\bysame}{\leavevmode\hbox to3em{\hrulefill}\thinspace}
\providecommand{\MR}{\relax\ifhmode\unskip\space\fi MR }
\providecommand{\MRhref}[2]{%
  \href{http://www.ams.org/mathscinet-getitem?mr=#1}{#2}
}
\providecommand{\href}[2]{#2}
\begin{thebibliography}{10}

\bibitem{Allaire-89}
G.~Allaire, \emph{Homogenization of the {S}tokes flow in a connected porous
  medium}, Asymptotic Anal. \textbf{2} (1989), no.~3, 203--222.

\bibitem{Allaire-91a}
\bysame, \emph{Continuity of the {D}arcy's law in the low-volume fraction
  limit}, Ann. Scuola Norm. Sup. Pisa Cl. Sci. (4) \textbf{18} (1991), no.~4,
  475--499.

\bibitem{Allaire-1997}
G.~Allaire and A.~Mikeli\'{c}, \emph{One-phase {N}ewtonian flow},
  Homogenization and porous media, Interdiscip. Appl. Math., vol.~6, Springer,
  New York, 1997, pp.~45--76, 259--275.

\bibitem{Armstrong-2018}
S.~Armstrong and P.~Dario, \emph{Elliptic regularity and quantitative
  homogenization on percolation clusters}, Comm. Pure Appl. Math. \textbf{71}
  (2018), no.~9, 1717--1849.

\bibitem{Armstrong-book}
S.~Armstrong, T.~Kuusi, and J.-C. Mourrat, \emph{Quantitative stochastic
  homogenization and large-scale regularity}, Grundlehren der Mathematischen
  Wissenschaften [Fundamental Principles of Mathematical Sciences], vol. 352,
  Springer, Cham, 2019.

\bibitem{Armstrong-2016}
S.~N. Armstrong and C.~Smart, \emph{Quantitative stochastic homogenization of
  convex integral functionals}, Ann. Sci. \'Ec. Norm. Sup\'er. (4) \textbf{49}
  (2016), no.~2, 423--481.

\bibitem{AL-1987}
M.~Avellaneda and F.~Lin, \emph{Compactness methods in the theory of
  homogenization}, Comm. Pure Appl. Math. \textbf{40} (1987), no.~6, 803--847.

\bibitem{Gloria-2021}
M.~Duerinckx and A.~Gloria, \emph{Quantitative homogenization theory for random
  suspensions in steady stokes flow}, preprint, arXiv:2103.06414 (2021).

\bibitem{FKV}
E.~B. Fabes, C.~E. Kenig, and G.~C. Verchota, \emph{The {D}irichlet problem for
  the {S}tokes system on {L}ipschitz domains}, Duke Math. J. \textbf{57}
  (1988), no.~3, 769--793.

\bibitem{Otto-Fisher}
J.~Fischer and F.~Otto, \emph{A higher-order large-scale regularity theory for
  random elliptic operators}, Comm. Partial Differential Equations \textbf{41}
  (2016), no.~7, 1108--1148.

\bibitem{Galdi}
G.~P. Galdi, \emph{An introduction to the mathematical theory of the
  {N}avier-{S}tokes equations}, second ed., Springer Monographs in Mathematics,
  Springer, New York, 2011, Steady-state problems.

\bibitem{Gia-book}
M.~Giaquinta, \emph{Multiple integrals in the calculus of variations and
  nonlinear elliptic systems}, Annals of Mathematics Studies, vol. 105,
  Princeton University Press, Princeton, NJ, 1983.

\bibitem{Gia}
M.~Giaquinta and G.~Modica, \emph{Nonlinear systems of the type of the
  stationary {N}avier-{S}tokes system}, J. Reine Angew. Math. \textbf{330}
  (1982), 173--214.

\bibitem{Gu-Shen}
S.~Gu and Z.~Shen, \emph{Homogenization of {S}tokes systems and uniform
  regularity estimates}, SIAM J. Math. Anal. \textbf{47} (2015), no.~5,
  4025--4057.

\bibitem{LA-1990}
R.~Lipton and M.~Avellaneda, \emph{Darcy's law for slow viscous flow past a
  stationary array of bubbles}, Proc. Roy. Soc. Edinburgh Sect. A \textbf{114}
  (1990), no.~1-2, 71--79.

\bibitem{Masmoudi-2004}
N.~Masmoudi, \emph{Some uniform elliptic estimates in a porous medium}, C. R.
  Math. Acad. Sci. Paris \textbf{339} (2004), no.~12, 849--854.

\bibitem{Mik-1991}
A.~Mikeli\'{c}, \emph{Homogenization of nonstationary {N}avier-{S}tokes
  equations in a domain with a grained boundary}, Ann. Mat. Pura Appl. (4)
  \textbf{158} (1991), 167--179.

\bibitem{Chase-Russell-2017}
B.~C. Russell, \emph{Homogenization in perforated domains and interior
  {L}ipschitz estimates}, J. Differential Equations \textbf{263} (2017), no.~6,
  3396--3418.

\bibitem{SP}
Enrique S\'{a}nchez-Palencia, \emph{Nonhomogeneous media and vibration theory},
  Lecture Notes in Physics, vol. 127, Springer-Verlag, Berlin-New York, 1980.

\bibitem{Shen-book}
Z.~Shen, \emph{Periodic homogenization of elliptic systems}, Operator Theory:
  Advances and Applications, vol. 269, Birkh\"{a}user/Springer, Cham, 2018,
  Advances in Partial Differential Equations (Basel).

\bibitem{Shen-1}
\bysame, \emph{Sharp convergence rates for {D}arcy law}, preprint,
  arXiv:2011.14169 (2020).

\bibitem{Shen-2020}
\bysame, \emph{{Large-scale Lipschitz estimates for elliptic systems with
  periodic high-contrast coefficients}}, Comm. Partial Diff. Eqs. (to appear).

\bibitem{Tartar-80}
L.~Tartar, \emph{Incompressible fluid flow in a porous medium - convergence of
  the homogenization process}, Non-Homogeneous Media and Vibration Theory,
  Lecture Notes in Physics, vol. 129, 1980, pp.~368--377.

\bibitem{Yeh-2010}
L.-M. Yeh, \emph{Elliptic equations in highly heterogeneous porous media},
  Math. Methods Appl. Sci. \textbf{33} (2010), no.~2, 198--223.

\bibitem{Yeh-2015}
\bysame, \emph{Pointwise estimate for elliptic equations in periodic perforated
  domains}, Commun. Pure Appl. Anal. \textbf{14} (2015), no.~5, 1961--1986.

\bibitem{Yeh-2016-1}
\bysame, \emph{{$L^p$} gradient estimate for elliptic equations with
  high-contrast conductivities in {$\Bbb{R}^n$}}, J. Differential Equations
  \textbf{261} (2016), no.~2, 925--966.

\bibitem{Yeh-2016-u}
\bysame, \emph{Uniform bound and convergence for elliptic homogenization
  problems}, Ann. Mat. Pura Appl. (4) \textbf{195} (2016), no.~6, 1803--1832.

\end{thebibliography}

\bigskip

\begin{flushleft}

Zhongwei Shen,
Department of Mathematics,
University of Kentucky,
Lexington, Kentucky 40506,
USA.

E-mail: zshen2@uky.edu
\end{flushleft}

\bigskip

\end{document}